\documentclass{conm-p-l}
\usepackage{amscd,amssymb,amsmath}
\usepackage[all]{xy}
\usepackage[colorlinks,plainpages,backref,urlcolor=blue]{hyperref}

\newtheorem{theorem}{Theorem}[section]
\newtheorem{corollary}[theorem]{Corollary}
\newtheorem{lemma}[theorem]{Lemma}
\newtheorem{prop}[theorem]{Proposition}

\theoremstyle{definition}
\newtheorem{definition}[theorem]{Definition}
\newtheorem{example}[theorem]{Example}
\newtheorem{remark}[theorem]{Remark}
\newtheorem{question}[theorem]{Question}

\newenvironment{romenum}
{ 

\begin{enumerate}}{\end{enumerate}}

\newcommand{\Z}{\mathbb{Z}}
\newcommand{\Q}{\mathbb{Q}}
\newcommand{\R}{\mathbb{R}}
\newcommand{\C}{\mathbb{C}}

\newcommand{\CP}{\mathbb{CP}}

\newcommand{\QP}{\mathbb{QP}}

\newcommand{\sV}{\mathsf{V}}
\newcommand{\sW}{\mathsf{W}}
\newcommand{\sE}{\mathsf{E}}

\renewcommand{\k}{\Bbbk}
\newcommand{\RR}{\mathcal{R}}
\newcommand{\VV}{\mathcal{V}}
\newcommand{\A}{{\mathcal{A}}}
\newcommand{\B}{{\mathcal{B}}}
\newcommand{\CC}{{\mathcal{C}}}
\newcommand{\WW}{\mathcal{W}}

\newcommand{\cV}{\mathcal{V}}

\newcommand{\G}{\Gamma}

\DeclareMathOperator{\rank}{rank}
\DeclareMathOperator{\gr}{gr}
\DeclareMathOperator{\im}{im}

\DeclareMathOperator{\id}{id}
\DeclareMathOperator{\ab}{{ab}}

\DeclareMathOperator{\Sym}{Sym}
\DeclareMathOperator{\ch}{char}
\DeclareMathOperator{\supp}{supp}
\DeclareMathOperator{\depth}{depth}
\DeclareMathOperator{\GL}{GL}
\DeclareMathOperator{\Hom}{{Hom}}

\DeclareMathOperator{\ann}{{ann}}

\DeclareMathOperator{\proj}{pr}

\DeclareMathOperator{\Spec}{{Spec}}

\DeclareMathOperator{\FP}{{FP}}

\DeclareMathOperator{\init}{in}
\DeclareMathOperator{\lk}{lk}

\DeclareMathOperator{\const}{const}

\DeclareMathOperator{\Grass}{Gr}
\DeclareMathOperator{\PD}{PD}

\newcommand{\PL}{\scriptscriptstyle{\rm PL}}

\newcommand{\m}{{\mathfrak m}}

\newcommand{\wX}{\widetilde{X}}

\newcommand{\same}{\Longleftrightarrow}
\newcommand{\surj}{\twoheadrightarrow}

\newcommand{\isom}{\xrightarrow{\,\simeq\,}}
\newcommand{\compl}{\scriptscriptstyle{\complement}}

\newcommand{\abs}[1]{\left| #1 \right|}

\def\angl#1{{\langle #1\rangle}}
\def\set#1{{\{ #1\}}}

\newcommand{\cv}{\check{{\mathcal V}}_1}
\newcommand{\phinorm}{\|\phi\|}
\newcommand{\chim}{\chi^{\:}_{_{-}}}

\begin{document}
%% revised: March 21, 2010 
%% revised again:  June 26, 2010
%% final revision:  October 19, 2010

\title[Fundamental groups and cohomology jumping loci]{%
Fundamental groups, Alexander invariants, and \\ cohomology jumping loci} 

\author[Alexander~I.~Suciu]{Alexander~I.~Suciu}
\address{Department of Mathematics,
Northeastern University,
Boston, MA 02115, USA}
\email{a.suciu@neu.edu}
\urladdr{http://www.math.neu.edu/\~{}suciu}
\thanks{Partially supported by NSA grant H98230-09-1-0012 
and an ENHANCE grant from Northeastern University}

\subjclass[2000]{Primary
14F35; %% Homotopy theory; fundamental groups
Secondary
20J05,  %% Homological methods in group theory
32S22,  %% Relations with arrangements of hyperplanes
55N25. %% Homology with local coefficients, equivariant cohomology
}

\keywords{Fundamental group, Alexander polynomial, 
characteristic variety, resonance variety, abelian cover, 
formality, Bieri--Neumann--Strebel--Renz invariant, 
right-angled Artin group, K\"{a}hler manifold, 
quasi-K\"{a}hler manifold, hyperplane arrangement, 
Milnor fibration, boundary manifold}

\dedicatory{Dedicated to Anatoly Libgober on the 
occasion of his sixtieth birthday}

\begin{abstract}

We survey the cohomology jumping loci and the 
Alexander-type invariants associated to a space,  
or to its fundamental group.  Though most of the 
material is expository, we provide new examples 
and applications, which in turn raise several 
questions and conjectures. 

The jump loci of a space $X$ come in two basic 
flavors: the characteristic varieties, or, the support 
loci for homology with coefficients in rank $1$ localÊ
systems, and the resonance varieties, or, the support lociÊ
for the homology of the cochain complexes arisingÊ
from multiplication by degree $1$ classes in theÊ
cohomology ring of $X$. ÊThe geometry of these 
varieties is intimately related to the formality, 
(quasi-) projectivity, and homological finitenessÊ
properties of $\pi_1(X)$.Ê

We illustrate this approach with various applications 
to the study of hyperplane arrangements, Milnor fibrations, 
$3$-manifolds, and right-angled Artin groups.  
\end{abstract}
\maketitle

\tableofcontents

\section{Introduction}
\label{sect:intro}

\subsection{Fundamental groups}
\label{intro:pi1}
The fundamental group of a topological space was introduced 
in 1904 by H.~Poincar\'{e}, with the express purpose of 
distinguishing between certain manifolds, such as the dodecahedral 
space and the three-sphere, which otherwise share a lot in common. 
Subsequently, it was realized that the geometric nature of a 
manifold $M$ influences the group-theoretic properties of its 
fundamental group, $\pi_1(M)$; and conversely, the nature 
of a group $G$ can say a lot about those manifolds with 
fundamental group $G$, at least in low dimensions.

As is well-known, every finitely presented group $G$ occurs 
as the fundamental group of a smooth, compact connected, 
orientable manifold $M$ of dimension $n=4$. The manifold 
can even be chosen to be symplectic (Gompf); 
and, if one is willing to go to dimension $n=6$, 
it can be chosen to be complex (Taubes).  
On the other hand, requiring $M$ to be 
$3$-dimensional, or to carry a K\"{a}hler structure 
puts severe restrictions on its fundamental group.

A finitely presented group $G$ is said to be a 
{\em K\"{a}hler group}\/ if it can be realized as 
$G=\pi_1(M)$, where $M$ is a compact, connected 
K\"{a}hler manifold. The notions of quasi-K\"{a}hler group 
and \mbox{(quasi-)} projective group are defined similarly. 
A classical problem, formulated by \mbox{J.-P.~Serre} 
in the late 1950s (cf.~\cite{Se}), is to determine which 
groups can be so realized.    

A basic theme of this survey is that certain topological 
invariants, such as the Alexander polynomial, or the 
cohomology jumping loci, are very good tools for 
studying questions about fundamental groups in 
algebraic geometry.  In particular, they can be brought 
to bear to settle Serre's realization problem for several 
notable classes of groups. 

\subsection{Alexander-type invariants}
\label{intro:alex}

In his foundational paper \cite{Al}, 
J. W. Alexander assigned to every knot $K$ in 
$S^3$ a polynomial with integer coefficients, as follows. 
Let $\pi\colon X'\to X$ be the infinite cyclic cover of 
the knot complement. Then 
$A=H_1(X', \pi^{-1}(x_0), \Z)$ is a finitely presented 
module over the group ring $\Z\Z\cong \Z[t^{\pm 1}]$.  
The Alexander polynomial, $\Delta_K(t)$, equals---up to 
normalization---the greatest common divisor of the 
codimension~$1$ minors of a presentation matrix for 
the Alexander module $A$.  Clearly, $\Delta_K(t)$ 
depends only on the knot group, $G=\pi_1(X,x_0)$.  
More generally, if $K$ is an $n$-component 
link in $S^3$, there is a multi-variable Alexander polynomial 
$\Delta_K(t_1, \dots , t_n)$, depending only on the link 
group, and a choice of meridians. 

These definitions extend with only slight modifications 
to arbitrary finitely presented groups (see \S\ref{sect:alex}). 
In \cite{Li82}, A.~Libgober considered the single-variable 
Alexander polynomial $\Delta_C(t)$ associated to the 
complement of an algebraic curve $C\subset \C^2$, and 
its total linking cover. He then proved in \cite{Li94} 
that all the roots of $\Delta_C(t)$ are roots of unity, 
a result which constrains the class of groups realizable 
as fundamental groups of plane curve complements. 

As shown in \cite{DPS-imrn, DPS-duke}, the multi-variable 
Alexander polynomial and the related Alexander 
varieties provide further constraints on the kind of 
fundamental groups that can appear in this algebraic 
context. For example, if $G$ is a quasi-projective group 
with $b_1(G)\ne 2$, then the Newton polytope of 
$\Delta_G$ is a line segment; and, if $G$ is 
a projective group, then $\Delta_G$ must be constant 
(see Theorem \ref{thm:alex qk}). 

\subsection{Cohomology jumping loci}
\label{intro:cvs}

An essential tool for us are the cohomology jumping loci 
associated to a connected, finite-type CW-complex $X$.  
The {\em characteristic varieties}\/ $\VV^i_d(X,\k)$ are 
algebraic subvarieties of  $\Hom (\pi_1(X), \k^{\times})$, 
while the {\em resonance varieties} $\RR^i_d(X,\k)$ 
are homogeneous subvarieties of  $H^1(X,\k)$.  
The former are the jump loci for the homology of 
$X$ with rank $1$ twisted complex coefficients 
(see \S\ref{sect:cvs}), while the latter are the jump 
loci for the homology of the cochain complexes 
arising from multiplication by degree $1$ classes in 
$H^*(X,\k)$ (see  \S\ref{sect:res vars}).  The jump 
loci of a group are defined in terms of the jump 
loci of the corresponding classifying space. 

It has been known since the work of Hironaka \cite{Hi97} 
that the degree $1$ characteristic varieties of a finitely 
generated group $G$ coincide with the determinantal 
varieties of its Alexander matrix, at least away from 
the origin.  A more general statement, valid in arbitrary 
degrees, was recently proved in \cite{PS-bns} 
(see Theorem \ref{thm:vw}). 

Foundational results on the structure of the characteristic 
varieties of K\"{a}hler and quasi-K\"{a}hler manifolds and 
smooth projective and quasi-projective varieties were 
obtained by Beauville \cite{Be}, Green--Lazarsfeld 
\cite{GL}, Simpson \cite{Sp92}, Campana \cite{Cm01}, 
and ultimately, Arapura~\cite{Ar}.  
In particular, if $X=\overline{X}\setminus D$, 
with $\overline{X}$ compact K\"{a}hler, $b_1(\overline{X})=0$, 
and $D$ a normal-crossings divisor, then each variety 
$\VV^i_d(X,\C)$ is a union of subtori of 
$\Hom (\pi_1(X) ,\C^{\times})$, possibly translated 
by unitary characters (see \S\ref{sect:kahler}).  
In \cite{Li01}, Libgober---who coined the name 
``characteristic varieties"---showed how to find irreducible 
components of $\VV^1_d(X,\C)$ from the faces of a 
certain ``polytope of quasiadjunction,'' in the case 
when $X$ is a plane curve complement.  Recently, 
he proved in \cite{Li09} a local version of Arapura's 
theorem, in which all the translations are done 
by characters of finite order. 

The resonance varieties were first defined by Falk 
in \cite{Fa97}, in the case when $X=X(\A)$ is the 
complement of a complex hyperplane arrangement $\A$.  
Best understood are the degree $1$ resonance varieties 
$\RR_d(\A)=\RR^1_d(X(\A),\C)$;  these varieties 
admit a very precise combinatorial description, 
owing to work of Falk \cite{Fa97}, Cohen--Suciu \cite{CS99}, 
Libgober \cite{Li01}, Libgober--Yuzvinsky \cite{LY}, and others, 
with the state of the art being the work of Falk, 
Pereira, and Yuzvinsky \cite{FY, PeY, Yu}. 

The cohomology jump loci provide a unifying framework for 
the study of a host of questions, both quantitative and 
qualitative, concerning a space $X$ and its fundamental 
group, $G=\pi_1(X)$.  A lot of the recent developments 
described in this survey come from our joint work 
with A.~Dimca and S.~Papadima \cite{DPS-imrn, 
DPS-duke, DPS-qk, DS, PS-toric, PS-bns}.  

\subsection{Abelian covers}
\label{intro:abel}
 
As shown by Libgober in \cite{Li92}, with further refinements 
by Hironaka \cite{Hi93}, Sarnak--Adams \cite{SarA}, 
Sakuma \cite{Sa}, and Matei--Suciu \cite{MS02}, counting 
torsion points on the character group, according 
to their depth with respect to the stratification
by the characteristic varieties, yields very precise 
information about the homology of finite abelian 
covers of $X$ (see~\S\ref{sect:abel covers}).  

An old observation in knot theory is that the Alexander 
polynomial $\Delta_K(t)$ of a fibered knot $K$ is monic.  
More generally, Dwyer and Fried \cite{DF} showed that 
the support varieties of the Alexander invariants of a 
finite CW-complex completely determine the homological 
finiteness properties of its free abelian covers.  This 
result was recast in \cite{PS-bns} in terms of the 
characteristic varieties and their (exponential) 
tangent cones (see Theorem \ref{thm:df cv}). 

In \cite{BNS}, Bieri, Neumann, and Strebel associated 
to every finitely generated group $G$ an open, conical 
subset  $\Sigma^1(G)$ of the real vector space $\Hom(G,\R)$. 
The BNS invariant and its higher-order generalizations, the 
invariants $\Sigma^q(G,\Z)$ of Bieri and Renz \cite{BR}, 
hold subtle information about the homological finiteness 
properties of normal subgroups of $G$ with abelian 
quotients (see \S\ref{sect:bnsr}). 
 The actual computation of the $\Sigma$-invariants 
is enormously complicated.  Yet, as shown in \cite{PS-bns}, 
the cohomology jumping loci of a classifying space $K(G,1)$ 
provide computable upper bounds for the $\Sigma$-invariants 
of $G$ (see Theorems \ref{thm:bns tau} and \ref{thm:bns valuation}). 

\subsection{Formality and the tangent cone formula}
\label{intro:tcf}

A central point in the theory of cohomology jumping loci 
is the relationship between the characteristic and resonance 
varieties.  As proved by Libgober in \cite{Li02}, the tangent 
cone to $\VV^i_d(X,\C)$ at the origin is contained in 
$\RR^i_d(X,\C)$, for any finite-type CW-complex $X$.  
But, as first noted in \cite{MS02}, the inclusion can be 
strict; in fact, as shown in \cite{DPS-duke}, equality may 
even fail for quasi-projective groups.  

The crucial property that bridges the gap between the tangent 
cone to a characteristic variety and the corresponding 
resonance variety is {\em formality} (see \S\ref{sect:formal}).  
A space $X$ as above is formal if its Sullivan minimal model 
is quasi-isomorphic to $(H^*(X,\Q),0)$, while a finitely 
generated group $G$ is $1$-formal if its Malcev Lie algebra 
has a quadratic presentation.  For a recent survey of these 
notions, we refer to \cite{PS-formal}.   

One of the main results from \cite{DPS-duke} establishes 
an isomorphism between the analytic germ of $\VV^1_d(G,\C)$ 
at $1$ and the analytic germ of $\RR^1_d(G,\C)$ at $0$, in the 
case when $G$ is a $1$-formal group (see Theorem \ref{thm:tcone}). 
This isomorphism provides a new, and very effective $1$-formality 
criterion. 

As also shown in \cite{DPS-duke}, the tangent 
cone theorem, when used in conjunction with Arapura's 
theorem, imposes very strong restrictions on the 
nature of the resonance varieties of K\"{a}hler groups, 
and, more generally, $1$-formal, quasi-K\"{a}hler groups 
(see Theorem \ref{thm:res kahler}).   These restrictions 
lead to both quasi-projectivity obstructions (in the formal 
setting), and enhanced formality obstructions (in the 
quasi-projective setting). 

\subsection{Applications}
\label{intro:apps}

The techniques described above have a large range of 
applicability.  We illustrate their usefulness with a variety 
of examples, arising in algebraic geometry, low-dimensional 
topology, combinatorics, and group theory. 

As a running example, we use an especially well-suited 
class of combinatorially defined spaces.  A simplicial 
complex $L$ on $n$ vertices determines a subcomplex 
$T_L$ of the $n$-torus, with fundamental 
group the right-angled Artin group $G_{\G}$ corresponding 
to the graph $\G=L^{(1)}$. The cohomology jumping loci 
of these ``toric complexes" were computed in \cite{PS-toric}.  
Using the computation of $\RR^1_1(G_\G,\C)$, it 
was shown in \cite{DPS-duke} that a group $G_\G$ 
is quasi-projective if and only if $\G$ is a complete 
multi-partite graph. We show here that the Alexander 
polynomial of $G_\G$ is non-constant if and only if 
$\G$ has connectivity $1$ (see Proposition 
\ref{prop:alexpoly raag}).   

In \S\ref{sect:arrs}, we treat in detailed fashion
an important class of quasi-projective varieties: those 
arising as complements of complex hyperplane arrangements. 
The cohomology jumping loci of such a complement, 
$X=X(\A)$, are rather well understood, especially 
in degree $1$, with a major open problem being 
the determination of the translated components  
in $\VV^1_1(X(\A),\C)$.  Much effort has been put 
over the years into comprehending fundamental 
groups of arrangements (see \cite{Su01} for more 
on this subject).   We identify here precisely the class 
of line arrangements $\A$ in $\C^2$ for which the 
group $G(\A)=\pi_1(X(\A))$ is a K\"{a}hler group 
(see Theorem \ref{thm:arr kahler}), respectively, 
a right-angled Artin group (see Theorem \ref{thm:arr raag}). 
Along the way, we reprove a recent result of Fan \cite{Fan09}, 
identifying those line arrangements $\A$ for which 
$G(\A)$ is a free group (see Theorem \ref{thm:fan}). 

In \S\ref{sect:milnor}, we discuss Milnor fibrations, with 
special emphasis on those arising from hyperplane arrangements. 
A well-known construction, due to J.~Milnor \cite{Mi}, associates 
to a weighted homogeneous polynomial $f\in \C[z_0,\dots, z_d]$ 
of degree $n$ a smooth fibration, $f\colon X \to \C^{\times}$, where 
$X$ is the complement in $\C^{d+1}$ of the hypersurface $\{f=0\}$. 
The Milnor fiber, $F=f^{-1}(1)$, is an $n$-fold cyclic cover 
of $U=X/\C^{\times}$. We review this construction, and the 
method of computing the homology groups $H_1(F,\k)$
from the characteristic varieties $\VV^1_d(U,\k)$, 
provided $(\ch \k, n)=1$.  We also review  
recent progress on the formality question for 
the Milnor fiber, first raised in \cite{PS-formal}, and 
recently answered---with different examples---by 
Zuber \cite{Zu} and Fern\'andez de Bobadilla~\cite{FdB}. 

We conclude in \S\ref{sect:3mfd} with a look at the world 
of $3$-dimensional manifolds. A major application of the 
techniques discussed in this survey is the solution given 
in \cite{DS} to the following question asked by 
S.~Donaldson and W.~Goldman in 1989, and 
independently by A.~Reznikov in 1993:  Which 
$3$-manifold groups are K\"{a}hler groups? 
The proof uses in an essential way the contrasting 
nature of the resonance varieties for these two 
classes of groups.  Another application, given 
in \cite{DPS-qk}, is to the classification (up to Malcev 
completion) of $3$-manifold groups which are also 
$1$-formal, quasi-K\"ahler groups.  This classification 
can be made very precise in the case of boundary 
manifolds of line arrangements in $\CP^2$: as 
shown in \cite{CS08} and \cite{DPS-imrn}, the 
only way the fundamental group of a boundary 
manifold $M(\A)$ can be either $1$-formal, or 
quasi-K\"ahler (or both), is for $\A$ to be a pencil, 
or a near-pencil. 

\section{Characteristic varieties}
\label{sect:cvs}

We start with the jumping loci for homology in rank $1$ 
local systems, and two types of tangent cones associated to a 
subvariety of the character variety of a group. 

\subsection{Rank~$1$ local systems}
\label{subsec:loc syst}   
Let $X$ be a connected CW-complex, with finite $k$-skeleton, 
for some $k\ge 1$.  Without loss of generality, we may assume 
$X$ has a single $0$-cell, call it $x_0$.  Moreover, we may 
assume all attaching maps $(S^i,*) \to (X^i,x_0)$ 
are basepoint-preserving. 

Fix a field $\k$, and denote by $(C_i(X, \k),\partial_i)_{i\ge 0}$ 
the cellular chain complex of $X$. Let $p\colon \wX \to X$ be the 
universal cover.  The cell structure on $X$ lifts in a natural fashion 
to a cell structure on $\wX$.  Fixing a lift $\tilde{x}_0\in p^{-1}(x_0)$ 
identifies the fundamental group, $G=\pi_1(X,x_0)$, with the 
group of deck transformations of $\wX$, which permute the cells. 
Therefore, we may view $(C_i(\wX, \k),\tilde{\partial}_i)_{i\ge 0}$  
as a chain complex of left-modules over the group ring $\k{G}$.  

Let $\k^{\times}$ be the group of units in $\k$. 
The group of $\k$-valued characters, $\Hom(G,\k^{\times})$, 
is an algebraic group, with pointwise multiplication inherited 
from $\k^{\times}$,  and identity the character taking constant 
value $1\in \k^{\times}$ for all $g\in G$. 
This character group parametrizes rank~$1$ local systems 
on $X$: given a character $\rho\colon G \to \k^{\times}$, 
denote by $\k_{\rho}$ the $1$-dimensional $\k$-vector space, 
viewed as a right module over the group ring $\k{G}$ 
via $a \cdot g = \rho(g)a$, for $g\in G$ and $a\in \k$. 
The homology groups of $X$ with coefficients in $\k_{\rho}$ 
are defined as 
\begin{equation}
\label{eq:twist hom}
H_i(X,\k_{\rho}):=H_i(C_{\bullet}(\wX,\k)\otimes_{\k{G}}\k_{\rho}).
\end{equation}
In this setup, the identity $1\in \Hom(G,\k^{\times})$ 
corresponds to the trivial coefficient system, $\k_{1}:=\k$, and  
$H_*(X,\k)$ is the usual homology of $X$ with $\k$-coefficients. 

\subsection{Homology jump loci}
\label{subsec:cvs} 

Computing homology groups with coefficients in rank $1$   
local systems leads to a natural filtration of the character 
group. 

\begin{definition}
\label{def:cvs}
The {\em characteristic varieties}\/ of $X$ (over $\k$) 
are the Zariski closed sets 
\begin{equation*}
\label{eq:cvs}
\VV^i_d(X,\k)=\{\rho \in \Hom(G,\k^{\times})
\mid \dim_\k H_i(X,\k_{\rho})\ge d\},
\end{equation*}
defined for all integers $0\le i\le k$ and $d>0$. 
\end{definition}

When computing the characteristic varieties in degrees up to $k$, 
we may assume, without loss of generality, that $X$ is a finite 
CW-complex of dimension $k+1$; see \cite[Lemma 2.1]{PS-bns} 
for an explanation. In each fixed degree $i$, the characteristic 
varieties define a descending filtration on the character group, 
\begin{equation}
\label{eq:filt}
\Hom(G,\k^{\times}) \supseteq \VV^i_1(X,\k^{\times}) \supseteq 
\VV^i_2(X,\k^{\times}) \supseteq \cdots.
\end{equation}

Clearly, $1\in \VV^i_1(X,\k)$ if and only if $H_i(X,\k)\ne 0$. 
In degree $0$, we have $\VV^0_1(X, \k)= \{ 1\}$ 
and $\VV^0_d (X, \k)=\emptyset$, for $d>1$.  In degree $1$, 
the characteristic varieties $\VV^1_d(X,\k)$ depend only on the 
fundamental group $G=\pi_1(X)$---in fact, only on its maximal 
metabelian quotient, $G/G''$---so we sometimes denote them 
as $\VV_d(G,\k)$.  

Define the {\em depth}\/ of a character 
$\rho\colon \pi_1(X,x_0)\to \k^{\times}$ relative to 
the stratification \eqref{eq:filt} by
\begin{equation} 
\label{eq:depth}
\depth^i_\k(\rho) = \max \{d \mid \rho \in \VV^i_d(X,\k)\}.
\end{equation}
As above, we abbreviate $\depth_\k (\rho)=\depth^1_\k(\rho)$. 

In general, the characteristic varieties depend on the field of 
definition $\k$.  Nevertheless, if $\k\subseteq \mathbb{K}$ is 
a field extension, then
\begin{equation}
\label{eq:vkk}
\VV^i_d(X,\k)=\VV^i_d(X,\mathbb{K}) \cap \Hom(G,\k^{\times}).
\end{equation}
Thus, for practical purposes it is usually assumed that $\k$ 
is algebraically closed. We will often suppress the coefficient 
field in the default situation when $\k=\C$; for instance, 
we will write $\VV^i_d(X)=\VV^i_d(X,\C)$. 

The depth-$1$ characteristic varieties satisfy a simple 
product formula:
\begin{equation}
\label{eq:cv prod}
\VV^i_1(X_1\times X_2,\k)= 
\bigcup_{p+q=i} \VV^{p}_1(X_1,\k) \times \VV^{q}_1(X_2,\k), 
\end{equation}
provided both $X_1$ and $X_2$ have finitely many 
$i$-cells, see \cite[Proposition 13.1]{PS-bns}. 

\subsection{Toric complexes}
\label{subsec:toric complexes}
Before developing the theory further, let us pause for 
some examples, showing how the characteristic varieties 
can be computed explicitly in favorable situations. 

\begin{example}
\label{ex:torus}
Let $S^1$ be the unit circle in $\R^2$, with basepoint $*=(1,0)$. 
Identify the fundamental group $\pi_1(S^1,x_0)$ with $\Z$, 
and its character group, $\Hom(\Z,\k)$, with $\k^{\times}$.  
It is readily seen that $\VV^0_1(S^1,\k)=\VV^1_1(S^1,\k)=\{1\}$, 
and $\VV^i_d(S^1,\k)  = \emptyset$, otherwise. 

More generally, let $T^n=S^1\times \cdots \times S^1$ be 
the $n$-torus.  Upon identifying 
$\pi_1(T^n)=\Z^n$ and $\Hom(\Z^n,\k)=(\k^{\times})^n$, 
we find that $\VV^i_d(T^n,\k)=\{1\}$, if $d\le \binom{n}{i}$, 
and $\VV^i_d(T^n,\k)  = \emptyset$, otherwise. 
\end{example}

\begin{example}
\label{ex:wedge}
Now let $(T^n)^{(1)}=\bigvee^{n} S^1$ be the $1$-skeleton 
of the $n$-torus, identified with the wedge of $n$ copies of $S^1$ at 
the basepoint.  Clearly, $\pi_1(\bigvee^{n} S^1)=F_n$, the free 
group of rank $n$, and $\Hom(F_n,\k)=(\k^{\times})^n$. 
It is readily seen that $\VV^1_d(\bigvee^{n} S^1,\k)=(\k^{\times})^n$ 
for $d\le n-1$, while $\VV^1_n(\bigvee^{n} S^1,\k)=\{1\}$ and 
$\VV^1_d(\bigvee^{n} S^1,\k)=\emptyset$ for $d>n$. 
\end{example}

The above computations can be put in an unified context, as follows. 
Given a simplicial complex $L$ on $n$ vertices, define the 
associated {\em toric complex}, $T_L$, as the subcomplex 
of the $n$-torus, obtained by deleting the cells corresponding 
to the missing simplices of $L$, i.e., 
\begin{equation}
\label{eq:toric complex}
T_L=\bigcup_{\sigma\in L}  T^{\sigma}, \quad 
\text{where $T^{\sigma}= \{ x \in T^n \mid x_i = * \text{ if } 
i\notin \sigma\}$}.
\end{equation}
This construction behaves well with respect to simplicial 
joins:  $T_{L*L'}=T_{L}\times T_{L'}$. 

Let $\G=(\sV,\sE)$ be the graph with vertex set $\sV$ the 
$0$-cells of $L$, and edge set $\sE$ the $1$-cells of $L$.  
The fundamental group of the toric complex $T_L$ is the 
{\em right-angled Artin group}
\begin{equation}
\label{eq:raag}
G_{\G} = \langle v \in \sV \mid vw = wv \text{ if } 
\{v,w\} \in \sE \rangle. 
\end{equation}
Groups of this sort interpolate between $G_{\G}=\Z^n$ in 
case $\G$ is the complete graph $K_n$, and $G_{\G}=F_n$ in 
case $\G$ is the discrete graph $\overline{K}_n$. Evidently, 
this class of groups is closed under direct products:  
$G_{\G}\times G_{\G'}=G_{\G*\G'}$. 

Given a right-angled Artin group $G_{\G}$, identify the 
character group $\Hom(G_{\G},\k^{\times})$ with the algebraic 
torus $(\k^{\times})^{\sV}:=(\k^{\times})^n$.  For each subset 
$\sW \subseteq \sV$, let $(\k^{\times})^{\sW} \subseteq 
(\k^{\times})^{\sV}$ be the corresponding subtorus; 
in particular, $(\k^{\times})^{\emptyset}=\{1\}$. 

\begin{theorem}[\cite{PS-toric}]
\label{thm:cjl tc}
With notation as above,
\begin{equation}
\label{eq:cv toric}
\VV^i_d(T_L,\k)=  \bigcup_{\stackrel{\sW \subseteq \sV}{%
\sum_{\sigma\in L_{\sV\setminus \sW}}
\dim_{\k} \widetilde{H}_{i-1-\abs{\sigma}} 
(\lk_{L_\sW}(\sigma),\k) \ge d}} (\k^{\times})^ \sW,
\end{equation}
where $L_\sW$ is the subcomplex induced by $L$ on $\sW$, 
and $\lk_{K}(\sigma)$ is the link of a simplex $\sigma$ 
in a subcomplex $K\subseteq L$. 
\end{theorem}

A classifying space for the group $G_{\G}$ is the toric complex 
$T_{\Delta_{\G}}$, where $\Delta_{\G}$ is the {\em flag complex}\/ 
of $\G$, i.e., the maximal simplicial complex with 
$1$-skeleton equal to the graph $\G$. Thus, for $i=d=1$, 
formula \eqref{eq:cv toric} yields:
\begin{equation} 
\label{eq:vv raag}
\VV_1(G_{\G},\k) = \bigcup_{\stackrel{\sW\subseteq \sV}{
\G_{\sW} \textup{ disconnected}}} (\k^{\times})^{\sW}.
\end{equation}

\subsection{Tangent cones}
\label{subsec:exp tcone}
We now return to the general situation from \S\ref{subsec:cvs}.   
In the sequel, we will be interested in approximating the 
characteristic varieties $\VV^i_d(X,\C)$ by their tangent 
cones (or exponential versions thereof) at the origin, that is, 
$1\in  \Hom(\pi_1(X,x_0),\C^{\times})$. 
We conclude this section with a review of these constructions. 

Let $G$ be a finitely generated group.   Its character group, 
$\Hom (G, \C^{\times})$, may be identified with the cohomology 
group $H^1(G, \C^{\times})$, while the group $\Hom (G, \C)$ 
may be identified with $H^1(G, \C)$.   The exponential map 
$\C\to \C^{\times}$, $z\mapsto e^z$ is a group homomorphism.  
As such, it defines a coefficient homomorphism, 
$\exp\colon H^1(G, \C) \to H^1(G, \C^{\times})$.

Now let $W$ be a Zariski closed subset of $\Hom (G, \C^{\times})$.   
The {\em tangent cone}\/ of $W$ at $1$ is the subset of $H^1(G,\C)=\C^n$, 
where $n=b_1(G)$, defined as follows.  Let $J$ be the ideal in the ring of 
analytic functions $\C \{ z_1,\dots , z_n\}$ defining the germ of $W$ 
at $1$, and let $\init (J)$ be the ideal in the polynomial ring 
$\C[z_1,\dots, z_n]$ spanned by the initial forms of non-zero 
elements of $J$. Then  
\begin{equation}
\label{eq:tc1}
TC_1(W)=V(\init (J)).
\end{equation}

On the other hand, the {\em exponential tangent cone}\/ of $W$ 
at $1$ is the set
\begin{equation}
\label{eq:tau1}
\tau_1(W)= \{ z\in H^1(G, \C) \mid \exp(tz)\in W,\ 
\text{for all $t\in \C$} \}.
\end{equation}

It is readily seen that both $TC_1(W)$ and $\tau_1(W)$ are 
homogeneous subvarieties of $H^1(G, \C)$, depending only 
on the analytic germ of $W$ at the identity.  In particular, 
$TC_1(W)$ and $\tau_1(W)$ are non-empty if and 
only if $1$ belongs to $W$.  

\begin{lemma}[\cite{DPS-duke}]
\label{lem:tau1 lin}
For any subvariety $W\subseteq H^1(G, \C^{\times})$, 
the exponential tangent cone $\tau_1(W)$ is a finite union 
of rationally defined linear subspaces of $H^1(G, \C)$.
\end{lemma}

Let us describe these subspaces explicitly.  
Clearly, $\tau_1$ commutes with intersections, 
so it is enough to consider the case $W=V(f)$, where 
$f= \sum_{u\in S} c_u t_1^{u_1}\cdots t_n^{u_n}$ 
is a non-zero Laurent polynomial, with support $S\subseteq \Z^n$. 
We may assume $f(1)=0$, for otherwise, $\tau_1(W)=\emptyset$. 
Let $\mathcal{P}$ be the set of partitions 
$p=S_1 \coprod \dots \coprod S_r$ of $S$, 
having the property that $\sum_{u\in S_i} c_u =0$, 
for  all $1\le i\le r$. For each such partition, let $L(p)$ 
be the (rational) linear subspace consisting of all points 
$z\in \C^n$ for which the dot product $\langle u-v, z\rangle$ 
vanishes, for all $u,v\in S_i$ and all $1\le i\le r$. 
Then, as shown in \cite[Lemma~4.3]{DPS-duke}, 
\begin{equation}
\label{eq:tau1w}
\tau_1(W)= \bigcup_{p\in \mathcal{P}} L(p).
\end{equation}

By \cite[Proposition 4.4]{PS-bns}, the following inclusion 
holds for any subvariety $W\subseteq H^1(G, \C^{\times})$:
\begin{equation}
\label{eq:tc incl}
\tau_1(W) \subseteq TC_1(W).
\end{equation}
If all irreducible components of $W$ containing $1$ 
are subtori, then clearly $\tau_1(W)=TC_1(W)$, 
but in general the two types of tangent cones differ. 

\section{Homology of abelian covers}
\label{sect:abel covers}

Much on the initial motivation for studying the characteristic varieties 
of a space comes from the precise information they give about the 
homology of its regular abelian covers. We now describe some of 
the  ways in which this is achieved. 

\subsection{Finite abelian covers}
\label{subsec:cyclic}
Work of Libgober \cite{Li92}, Hironaka \cite{Hi93, Hi97}, 
Sarnak and Adams \cite{SarA}, and Sakuma \cite{Sa} revealed 
the varied and fruitful connections between the characteristic 
varieties, the Alexander invariants (see \S\ref{sect:alex} below), 
and the Betti numbers of finite abelian covers.  Let us summarize 
two of those results, in a somewhat stronger form, following the 
treatment from \cite{MS02} and \cite{Su01}. 

As before, let $X$ be a connected CW-complex with finite 
$1$-skeleton, and denote by $G$ its fundamental group.  
Let $p\colon Y\to X$ be a regular, $n$-fold cyclic cover, 
with classifying map $\lambda\colon G \surj \Z_n$. Fix an 
algebraically closed field $\k$, and assume $\ch\k\nmid n$. 
Then, the homomorphism $\iota\colon \Z_n \to \k^{\times}$ 
which sends a generator of $\Z_n$ to a primitive $n$-th 
root of unity in $\k$ is an injection. For each integer $j>0$, 
define a character $(\iota\circ \lambda)^j\colon G\to \k^{\times}$ 
by $(\iota\circ \lambda)^j(g)=\iota(\lambda(g))^j$.  
Proceeding as in the proof of Theorem 6.1 from \cite{MS02}, 
we obtain the following slightly more general result.

\begin{theorem}[\cite{MS02}]
\label{thm:betti1 cover}
Let $\lambda\colon G \surj \Z_n$ be an epimorphism, 
and let $Y\to X$ be the corresponding regular cover.  
If $\bar{\k}=\k$ and $\ch\k\nmid n$, then
\begin{equation}
\label{eq:betti1}
\dim_{\k} H_1(Y, \k) = \dim_{\k} H_1(X,\k) + \sum_{1\ne k | n} 
\varphi(k) \depth_{\k} \big((\iota\circ \lambda)^{n/k}\big), 
\end{equation}
where $\varphi$ is the Euler totient function.  
\end{theorem}

In the same vein, for each $n>1$, let $X_n$ be the $n$-th congruence 
cover of $X$, i.e., the regular cover  determined by the canonical 
projection $\pi_1(X,x_0)\surj H_1(X,\Z_n)$.  Then, a mild   
generalization of Theorem 5.2 from \cite{Su01} yields the 
following formula for the first Betti number of such a cover:
\begin{equation}
\label{eq:b1cong}
b_1(X_n)=b_1(X)+\sum_{\rho} \depth_{\C}(\rho),
\end{equation}
where the sum is taken over all characters 
$\rho\in \Hom(G,\C^{\times})$ of order exactly equal to $n$. 

\subsection{Free abelian covers}
\label{subsec:free abel}

It turns out that the characteristic varieties also control 
the homological finiteness properties of (regular) free abelian 
covers. 

Let $\nu\colon G\surj \Z^r$  ($r>0$) be an epimorphism, 
and let $X^{\nu}\to X$ be the corresponding cover.  
Denote by $\nu^*\colon \Hom(\Z^r,\k^{\times} ) \to \Hom(G,\k^{\times} )$ 
the induced homomorphism between character groups. When   
$r=1$, denote by $\nu_{\k}\in H^1(X,\k)$ the corresponding 
cohomology class. 

\begin{theorem}[\cite{DF}, \cite{PS-bns}]
\label{thm:df cv}
Suppose $X$ has finite $k$-skeleton. 
\begin{romenum}
\item \label{df1} 
$\sum_{i=0}^{k} \dim_{\k} H_{i} (X^{\nu}, \k) <\infty \same 
\text{$\im(\nu^*) \cap \Big( \bigcup_{i=0}^{k} 
\VV^i_1(X,\k)\Big)$ is finite}$.
\item \label{df2} If $r=1$, then 
$\sum_{i=0}^{k} \dim_{\C} H_{i} (X^{\nu}, \C) <\infty \same 
\nu_{\C}\not\in  \bigcup_{i=0}^{k} \tau_1\big(\VV^i_1(X,\C)\big)$.
\end{romenum}
\end{theorem}

Part \eqref{df1} was proved by Dwyer and Fried in \cite{DF}, 
and reinterpreted in this context in \cite[Corollary 6.2]{PS-bns}, 
while part \eqref{df2} was proved in \cite[Theorem 6.5]{PS-bns}.  

An epimorphism $\nu\colon G\surj \Z^r$ as above gives rise to 
a map $\bar\nu\colon H_1(X,\Q)\surj \Q^r$, which may be viewed as 
an element in the Grassmanian $\Grass_r(H^1(X,\Q))$ 
of $r$ planes in the vector space $H^1(X,\Q)$. Following \cite{DF}, 
consider the set 
\begin{equation}
\label{eq:grass}
\Omega^k_r(X)=\Big\{\bar\nu\in \Grass_r(H^1(X,\Q)) \mid 
\sum_{i=0}^{k} \dim_{\C} H_{i} (X^{\nu}, \C) <\infty \Big\}. 
\end{equation}
By Theorem \ref{thm:df cv}\eqref{df2} and Lemma \ref{lem:tau1 lin}, 
$\Omega^k_1(X)$ is the complement of a finite union of projective 
subspaces; in particular, it is an open set. For $r>1$, though, 
$\Omega^k_r(X)$ is not necessarily open, as the following example 
of Dwyer and Fried shows.

\begin{example}[\cite{DF}]
\label{ex:df}
Let $Y = T^3 \vee S^2$.  Then $\pi_1(Y)=H$, a free abelian 
group on generators $x_1,x_2,x_3$, and $\pi_2 (Y) = \Z{H}$, 
a free module generated by the inclusion of $S^2$ in $Y$. 
Let $f\colon S^2 \to  Y$ represent the element $x_1-x_2+1$ 
of $\pi_2(Y)$, and attach a $3$-cell to $Y$ along $f$ to obtain 
a CW-complex $X=Y \cup_f D^3$, with $\pi_1(X) = H$ and 
$\pi_2(X)=\Z{H}/(x_1-x_2 +1)$.  Identifying 
$\Hom(H,\C)=(\C^{\times})^3$, we have $\VV^1_1(X)= \{1\}$ and 
$\VV^2_1(X)=\set{ z \in  (\C^{\times})^3 \mid z_1-z_2+1= 0}$. 

Consider an algebraic $2$-torus $T$ in $(\C^{\times})^3$, given 
by an equation of the form $\{z_1^{a_1}z_2^{a_2}z_3^{a_3}=1\}$, 
for some $a_i\in \Z$.  It is readily checked that the subvariety 
$T\cap \VV^2_1(X)$ is either the empty set (this happens 
precisely when  $T=\{z_1z_2^{-1}= 1\}$ or $T=\{z_2 = 1\}$),  
or is $1$-dimensional. Thus, the locus in the Grassmannian 
of $2$-planes in $H^1(X,\Q)=\Q^3$ giving rise to algebraic 
$2$-tori in $(\C^{\times})^3$ having finite intersection with 
$\VV^2_1(X)$ consists of two points.  It follows that exactly 
two $\Z^2$-covers of $X$ have finite Betti numbers. In particular, 
the set $\Omega^2_2(X)$ is not open in $\Grass_2(H^1(X,\Q))=\QP^2$, 
even for the usual topology on this projective plane.
\end{example}

For an in-depth discussion of the interplay between the 
geometry of the cohomology jumping loci and the homological 
finiteness properties of free abelian covers, we refer to \cite{Su10a} 
and \cite{Su10b}.

\section{Alexander invariants} 
\label{sect:alex}

In this section, we discuss the various Alexander-type invariants 
associated to a space, with special emphasis on the Alexander 
polynomial and the Alexander varieties, and how these objects 
relate to the characteristic varieties. 

\subsection{Alexander modules}
\label{subsec:alex modules}
As before, let $X$ be a connected CW-complex, with a unique 
$0$-cell, $x_0$, which we take as the basepoint, and with finitely 
many $1$-cells.  Let $G=\pi_1(X,x_0)$ be the fundamental group, 
and let 
\begin{equation}
\label{eq:tf ab}
H=H_1(G,\Z)/\text{torsion}\cong \Z^{b_1(G)}
\end{equation} 
be its maximal torsion-free abelian quotient. 
The canonical projection, $\ab\colon G \surj H$, determines 
a regular cover, $\pi\colon X' \to X$.  Set $F=\pi^{-1}(x_0)$. 
The exact sequence of the pair $(X', F)$ yields an exact 
sequence of $\Z{H}$-modules,
\begin{equation}
\label{eq:crow1}
\xymatrixcolsep{16pt}
\xymatrix{0 \ar[r]& H_1(X', \Z) \ar[r]& H_1(X',F, \Z) \ar[r]& 
H_0(F, \Z) \ar[r]& H_0(X', \Z) \ar[r]& 0}.
\end{equation}

The module $H_1(X', \Z)$ is called the {\em (first) Alexander invariant}\/ 
of $X$, while $H_1(X',F, \Z)$ is called the {\em Alexander module} 
of $X$.  Clearly, these two $\Z{H}$-modules depend only on the 
fundamental group $G=\pi_1(X,x_0)$, so we denote them by 
$B_G$ and $A_G$, respectively. Identifying the kernel of the 
map $H_0(F, \Z) \to H_0(X', \Z)$ with the augmentation ideal, 
$I_{H}=\ker \big(\epsilon\colon \Z{H} \to \Z\big)$, one extracts 
from \eqref{eq:crow1} the exact sequence 
$0 \to B_G \to A_G \to I_{H} \to 0$.

\subsection{Fox derivatives and the Alexander matrix} 
\label{subsec:alex matrix}
The  free differential calculus of R.~Fox \cite{Fox54, Fox56} 
yields an efficient algorithm for computing the Alexander module 
of a finitely generated group $G$.  For all practical purposes, 
we may assume $G$ is finitely presented; as explained in 
\cite[\S 2.6]{DPS-imrn}, there is no real loss of generality 
in doing that. 

Let $F_{q}$ be the free group with generators 
$x_1,\dots,x_{q}$.  For each $1\le j \le q$, there is a
linear operator $\partial_j=\partial/\partial x_j \colon \Z F_{q}\to 
\Z F_{q}$, known as the $j$-th Fox derivative, uniquely determined 
by the following rules: 
$\partial_j (1)=0$,
$\partial_j (x_i)=\delta_{ij}$, and
$\partial_j (uv)=\partial_j (u) \epsilon(v)+
u\partial_j (v)$.  

Next, let $G=\langle x_1,\dots ,x_{q}\mid r_1,\dots ,r_m\rangle$
be a finite presentation for our group, and let $\phi\colon F_q\to G$ 
be the presenting homomorphism.  Define the {\em Jacobian matrix}\/ 
of (the given presentation of) $G$ as 
\begin{equation}
\label{eq:jac mat}
J_G=\big(\partial_j r_i \big)^{\phi} 
\colon \Z{G}^m \to \Z{G}^q.
\end{equation}
As noted by Fox (see \cite[Theorem 4.1]{MS02} for a short proof), 
the matrix $J_G$ determines the first homology group 
of any finite-index subgroup of $G$.  

\begin{prop}[\cite{Fox56}]
\label{prop:fox}
Let $K<G$ be a subgroup of index $k<\infty$, 
let $\sigma\colon G\to \Sym(G/K)\cong S_k$ be the coset 
representation, and let $\pi\colon S_{k}\to \GL(k,\Z)$ be 
the permutation representation. Then $J_{G}^{\pi\circ \sigma}$ is 
a presentation matrix for the abelian group $H_1(K)\oplus\Z^{k-1}$.
\end{prop}

Define now the {\em Alexander matrix}\/ of the group $G$ as the 
(torsion-free) abelianization of the Fox Jacobian matrix of $G$, 
\begin{equation}
\label{eq:alex mat}
\Phi_G=J_G^{\ab} \colon \Z{H}^m \to \Z{H}^q. 
\end{equation}

\begin{prop}[\cite{Fox54}]
\label{prop:fox alex} 
The Alexander matrix $\Phi_G$ is a presentation matrix 
for the Alexander module $A_{G}$. 
\end{prop}

\subsection{Elementary ideals}
\label{subsec:fitting}
Before proceeding, we need to recall some basic notions 
from commutative algebra.  Let $R$ be a commutative ring 
with unit. Assume $R$ is Noetherian and a unique factorization 
domain.  Let $M$ be a finitely generated $R$-module.  
Then $M$ admits a finite presentation, of the form 
\begin{equation}
\label{eq:pres M}
\xymatrix{R^{m} \ar^{\Phi}[r] & R^{q} \ar[r] & M \ar[r] &0 }.
\end{equation}

The {\em $i$-th elementary ideal}\/ of $M$, denoted $E_i(M)$, 
is the ideal generated by the minors of size $q-i$ of the $q \times m$ 
matrix $\Phi$, with the convention that $E_i(M)=R$ if  $i \ge q$, 
and $E_i(M)=0$ if $q-i>m$. It is a standard exercise to show that 
$E_i(M)$ does not depend on the choice of presentation \eqref{eq:pres M}. 
Clearly, $E_i(M) \subset E_{i+1}(M)$, for all $i\ge 0$.  The annihilator 
ideal, $\ann (M)$, contains the ideal of maximal minors, $E_0(M)$, 
and they both have the same radical. 

Let $\Delta_i(M)$ be a generator of the smallest principal 
ideal in $R$ containing $E_i(M)$, i.e., the greatest common
divisor of all elements of $E_i(M)$. As such, $\Delta_i(M)$ is 
well-defined only up to units in $R$. (If two elements 
$\Delta$, $\Delta'$ in $R$ generate the same principal ideal, 
that is, $\Delta =u\Delta'$, for some unit $u\in R^*$, we 
shall write $\Delta \doteq \Delta'$.)   Note that 
$\Delta_{i+1}(M)$ divides $\Delta_i(M)$, for all $i\ge 0$. 

\subsection{Alexander polynomial} 
\label{subsec:alex poly}
As before, let $G$ be a finitely generated group, 
and let $H=\ab(G)$ be its maximal torsion-free abelian 
quotient.  It is readily seen that the group ring $\Z{H}$ is a 
(commutative) Noetherian ring, and a unique factorization 
domain.  

\begin{definition}
\label{def:alex poly}
The {\em Alexander polynomial}\/ of the group $G$ is 
the greatest common divisor of all elements in the 
first elementary ideal of the Alexander module of $G$,
\begin{equation}
\label{eq:alex G}
\Delta_G= \Delta_1(A_G)=\gcd(E_1(A_G))\in \Z{H}.
\end{equation}  
\end{definition}
 
From the discussion in \S\ref{subsec:fitting}, it follows that 
$\Delta_G$ depends only on $G$, modulo units in $\Z{H}$.  

Now suppose $G$ admits a finite presentation, with  
$q$ generators and $m$ relators.  Fix a basis 
$\alpha=\{\alpha_1,\dots ,\alpha_n\}$ 
for the free abelian group $H=\ab(G)$. 
This identifies the group ring $\Z{H}$ with the Laurent 
polynomial ring $\Lambda=\Z[t^{\pm 1}_1, \dots ,t_n^{\pm 1}]$. 
In this fashion, the Alexander polynomial of $G$ 
may be viewed as a Laurent polynomial in $n$ 
variables, well-defined up to monomials of the 
form $u=\pm  t_1^{\nu_1} \cdots t_n^{\nu_n}$.  
By Proposition \ref{prop:fox alex}, $\Delta_G$ is the 
greatest common divisor of the minors of size $q-1$ 
of the Alexander matrix, $\Phi_G\colon \Lambda^m\to \Lambda^q$. 

\begin{remark}
\label{rem:alex def}
Note that the Alexander polynomial, when 
viewed as an element of $\Lambda$, depends on the 
choice of basis $\alpha$ for $H=\Z^n$,  even though 
we suppress that dependence from the notation.  
If $\alpha'$ is another choice of basis, then the 
corresponding polynomial, $\Delta'_G$, is obtained 
from $\Delta_G$ by applying the automorphism of the ring 
$\Lambda=\Z\Z^n$ induced by the linear automorphism  
of $\Z^n$ taking $\alpha$ to $\alpha'$. Nevertheless, 
various features  of the Alexander polynomial, such 
as the constancy of $\Delta_G$, or the dimension 
of its Newton polytope, are unaffected by the choice 
of basis for $H$. 
\end{remark}

\subsection{Alexander varieties}
\label{subsec:alex vars}
Let $X$ be CW-complex as in \S\ref{subsec:alex modules}, 
with fundamental group $G=\pi_1(X,x_0)$, and let $X' \to X$ 
be the maximal abelian cover, defined by the projection 
$\ab\colon G \surj H$.  
Fix a field $\k$, and let $\Hom(G,\k^{\times})^0$ be the identity 
component of the algebraic group $\Hom(G,\k^{\times})$. 
Clearly, the map $\ab$ induces an isomorphism 
$\ab^*\colon \Hom(H,\k^{\times}) \isom \Hom(G,\k^{\times})^0$.

Now consider the $i$-th Alexander invariant of $X$, defined 
as the homology group $H_i(X',\k)$, viewed as a module over 
the group ring $\k{H}$. The support loci of the elementary ideals 
of this finitely generated $\k{H}$-module define a filtration 
of the character group $\Spec \k{H}=\Hom(H,\k^{\times})$. 

\begin{definition}
\label{def:alex}
The {\em Alexander varieties}\/ of $X$ (over $\k$) are 
the subvarieties of $\Hom(G,\k^{\times})^0$ given by 
\begin{equation*}
\label{eq:supp loci}
\WW^i_d(X,\k)=\ab^*(V(E_{d-1}(H_i(X',\k)))).
\end{equation*}
In particular, $\WW^i_1(X,\k)=\ab^*(V(\ann H_i(X',\k)))$.  
\end{definition}

In degree $i=1$, these varieties depend only on the group 
$G=\pi_1(X,x_0)$.   Suppose $G$ admits a finite presentation, 
with, say, $q$ generators, and identify the group $\Hom(G,\k^{\times})^0$ 
with the algebraic torus $(\k^{\times})^{b_1(G)}$.  The varieties 
$\WW_d(G,\k)=\WW^1_d(X,\k)$ can then be described as 
the subvarieties of this torus, defined by the vanishing 
of all minors of size $q-d$ of the Alexander matrix $\Phi_G$.

\subsection{Alexander varieties and characteristic varieties}
\label{subsec:av cv} 
It has been known since the work of Hironaka \cite{Hi97} 
that the characteristic and Alexander varieties of a space are 
closely related.  As shown in \cite[Theorem 3.6]{PS-bns}, there 
is an exact match between the filtrations defined on the identity 
component of the character variety by these subvarieties, at least 
for depth $d=1$.

\begin{theorem}[\cite{PS-bns}]
\label{thm:vw}
For all $q\ge 0$, 
\begin{equation}
\label{eq:vw}
\Hom(G,\k^{\times})^0 \cap\Big( \bigcup_{i=0}^{q} \VV^i_{1}(X,\k) \Big) = 
\bigcup_{i=0}^{q} \WW^i_{1}(X,\k).
\end{equation}
\end{theorem}

In homological degree~$1$, there is a more precise 
comparison, valid for arbitrary depths.  The following 
result was proved in \cite{Hi97}, 
and further extended in \cite{MS02} and \cite{DPS-imrn}. 

\begin{prop} 
\label{prop:echar}
Let $\rho\colon H\to \k^{\times}$ be a non-trivial character.  
Then, for all $d\ge 1$, 
\begin{equation*}
\label{eq:vars ab}
\ab^*(\rho) \in \VV_{d}(G,\k) \same
\rho \in V (E_{d}(A_{G}\otimes \k)) \same
\rho\in  V (E_{d-1}(B_{G}\otimes \k)).
\end{equation*}
\end{prop}

We now isolate a class of groups $G$ for which 
all characteristic varieties $\VV_{d}(G,\k)$  
are determined by the Alexander matrix $\Phi_G$, 
even at the trivial character. 

\begin{definition}
\label{def:comm rel}
A group $G$ is said to be a {\em commutator-relators}\/ 
group if it admits a presentation of the form 
$G=\langle x_1,\dots ,x_{n}\mid r_1,\dots ,r_m\rangle$, 
with each relator $r_i$ belonging to the commutator 
subgroup of $F_n=\langle x_1,\dots ,x_{n}\rangle$. 
\end{definition}

For a commutator-relators group $G$ as above, the group 
$H=H_1(G,\Z)$ is free abelian of rank $n$, and comes  
endowed with a preferred basis, $\ab(x_1),\dots ,\ab(x_{n})$. 
This allows us to identify in a standard way 
$\Hom(G,\k^{\times})$ with $(\k^{\times})^n$. 
Let $\Phi_G\colon \Lambda^m \to \Lambda^n$ 
be the Alexander matrix of $G$, and let 
$\Phi_G(t)\colon \k^m \to \k^n$ be its evaluation at a 
point $t=(t_1,\dots , t_n)$ in $(\k^{\times})^n$.  
A routine computation with Fox derivatives shows 
that $\Phi_G(1)$ is the zero matrix.  The next result 
then follows from Propositions \ref{prop:fox alex} 
and \ref{prop:echar}. 

\begin{prop}
\label{prop:alex comm}
Let $G$ be a commutator-relators group, with $b_1(G)=n$. 
Then $\VV_d(G,\k)=  \set{ t  \in (\k^{\times})^n \mid 
\rank \Phi_G(t) < n-d}$. 
\end{prop}

\subsection{Alexander polynomial and characteristic varieties}
\label{subsec:alex cvar}

Let $\VV_1(G)=\VV_1(G,\C)$ be the depth~$1$ 
characteristic variety of $G$, with coefficients in $\C$.  
Denote by $\cv (G)$ the union of all 
codimension-one irreducible components of 
$\VV_1(G) \cap \Hom(G,\C^{\times})^0$.  
The Alexander polynomial $\Delta_G$  
defines a hypersurface, $V(\Delta_G)$, in the complex 
algebraic torus $\Hom(G,\C^{\times})^0$.  
The next theorem details the relationship 
between $\cv (G)$ and $V(\Delta_G)$. 

\begin{theorem}[\cite{DPS-imrn}]
\label{thm:deltacd1}
For a finitely generated group $G$, the following hold:
\begin{romenum}
\item \label{dc1}
$\Delta_G=0$ if and only if 
$\Hom(G,\C^{\times})^0 \subseteq \VV_1(G)$.
In this case, $\cv (G)=\emptyset$.
\item \label{dc2}
If $b_1(G)\ge 1$ and $\Delta_G\ne 0$, then
\begin{equation*}
\cv (G) =\begin{cases}
V(\Delta_G) & \text{if $b_1(G)>1$}\\
V(\Delta_G)\coprod \{ 1\}  & \text{if $b_1(G)=1$}.
\end{cases}
\end{equation*}
\item \label{dc3} If $b_1(G)\ge 2$, then 
$\cv (G)=\emptyset$ if and only if $\Delta_G\doteq \const$.
\end{romenum}
\end{theorem}

In particular, if $\Delta_G$ does not vanish identically, then 
$\cv (G)\setminus\set{1}= V(\Delta_G) \setminus\set{1}$.  
Moreover, under all circumstances, 
\begin{equation}
\label{eq:v1 alex}
\VV_1(G)\setminus\set{1} \supseteq V(\Delta_G) \setminus\set{1}.
\end{equation}

In general, we cannot expect a perfect match between 
$\VV_1(G)$ and $V(\Delta_G)$, as the following example 
from knot theory illustrates.

\begin{example}
\label{ex:knot}
Let $K$ be a non-trivial knot in the $3$-sphere, with 
complement $X=S^3 \setminus K$. The knot group,  
$G=\pi_1(X)$, has abelianization $\Z$, and thus 
$1\in \VV_1(G)$.  On the other hand, the Alexander 
polynomial of the knot, $\Delta_G\in \Z[t^{\pm 1}]$,  
satisfies $\Delta_G(1)=\pm 1$; thus, $1\notin V(\Delta_G)$.  
Nevertheless, by Theorem \ref{thm:deltacd1}\eqref{dc2}, 
$\VV_1(G)\setminus\set{1}$ equals $V(\Delta_G)$, 
the set of roots of the Alexander polynomial. 
\end{example}

As a further application of Theorem \ref{thm:deltacd1}, 
we consider the Alexander polynomial $\Delta_{G_\G}$ 
associated to a right-angled Artin group, and ask:  
for which graphs $\G$ is $\Delta_{G_{\G}}$ constant? 
It is easy to see that $\Delta_{F_n}=0$, for $n\ge 1$, 
while $\Delta_{\Z^n}\doteq 1$, for $n>1$.  
To handle the general case, we need to 
recall a definition from graph theory.  

\begin{definition}
\label{def:graph conn}
The {\em connectivity}\/ of a graph $\G=(\sV,\sE)$, denoted 
$\kappa(\G)$, is the maximum integer $r$ so that, for any 
subset $\sW \subset \sV$ with $\abs{\sW}<r$, the induced  
subgraph on $\sV\setminus \sW$ is connected. 
\end{definition}

\begin{prop}
\label{prop:alexpoly raag}
A  right-angled Artin group $G_{\G}$ has non-constant 
Alexander polynomial  if and only if the 
graph $\G$ has connectivity $1$:
\[
\Delta_{G_{\G}} \not\doteq \const \same \kappa(\G) =1 .
\]
\end{prop}
\begin{proof}
Recall from formula \eqref{eq:vv raag} that 
$\VV_1(G_{\G})$ consists of coordinate subspaces 
$(\C^{\times})^{\sW}$, indexed by (maximal) subsets 
$\sW \subset \sV$ such that $\G_{\sW}$ is disconnected. 
Thus, $\cv(G_{\G})$ is non-empty if and only if $\G$ is 
connected and has a cut point, i.e., $\kappa(\G) =1$. 

If $\G$ has just $1$ vertex, then $\kappa(\G) =0$; on the 
other hand, $G_{\G}=\Z$, and so $\Delta_{G_{\G}}=0$. 
For all other graphs, $b_1(G_{\G})\ge 2$, and 
Theorem \ref{thm:deltacd1}\eqref{dc3} yields the 
desired conclusion.  
\end{proof}

\subsection{Almost principal Alexander ideals}
\label{subsec:almost princ}

We conclude this section with a class of groups $G$  for which 
the Alexander polynomial $\Delta_G$ may be used to inform 
in a more precise fashion on the characteristic varieties 
$\VV_d(G)$.  We start with a definition from \cite{DPS-imrn}, 
inspired by work of Eisenbud--Neumann \cite{EN} and 
McMullen \cite{McM}. 

\begin{definition}
\label{def:alexprinc}
Let $G$ be a finitely generated group, and set $H=\ab(G)$.  
We say that the Alexander ideal $E_1(A_G)\subset \Z{H}$ is 
{\em almost principal}\/ if there exists an integer $q\ge 0$ 
such that the following inclusion holds in $\C{H}$:
\begin{equation}
\label{eq:almost princ}
I_H^q\cdot ( \Delta_G )\subseteq E_1(A_G) \otimes \C.
\end{equation}
\end{definition}

For this class of groups, the Alexander polynomial determines 
to a large extent the depth~$1$ characteristic variety. 

\begin{prop}
\label{prop:v1 ap}
Suppose the Alexander ideal $E_1(A_G)$ is almost principal. Then:

\begin{romenum}
\item \label{al1}
$\big( \VV_1(G) \cap \Hom(G,\C^{\times})^0\big) \setminus\set{1} 
= V(\Delta_G) \setminus\set{1}$.
\item  \label{al2}
If, moreover, $H_1(G,\Z)$ is torsion-free, then 
$\VV_1(G)  \setminus\set{1} = V(\Delta_G) \setminus\set{1}$.
\end{romenum}
\end{prop}

\begin{proof}
Part \eqref{al1} follows from \eqref{eq:v1 alex} and 
\eqref{eq:almost princ}. Part \eqref{al2} is now obvious. 
\end{proof}l

Finally, let us remark on the connection between the multiplicities 
of the factors of $\Delta_G$ and the higher-depth characteristic 
varieties of $G$.  Identify $\Hom(G,\C^{\times})^0=
(\C^{\times})^n$, where $n=b_1(G)$.  For a character 
$\rho\in (\C^{\times})^n$, and a Laurent polynomial 
$f\in \Z[t^{\pm 1}_1, \dots ,t_n^{\pm 1}]$, denote by 
$\nu_{\rho}(f)$ the order of vanishing of the germ of 
$f$ at $\rho$.  

\begin{theorem}[\cite{DPS-imrn}]
\label{thm:h1 bound}
Suppose the Alexander ideal of $G$ is almost principal, and let 
$\Delta_G \doteq c f_1^{\mu_1}\cdots f_s^{\mu_s}$ be 
the decomposition into irreducible factors of the Alexander polynomial.
Then, $\dim_{\C} H_1(G; \C_{\rho})\le \sum_{j=1}^s \mu_j\cdot 
\nu_{\rho}(f_j)$, for all $\rho\in \Hom(G,\C^{\times})^0 \setminus \{1\}$. 
\end{theorem}

If the upper bound from Theorem \ref{thm:h1 bound} is attained 
for every nontrivial character $\rho\colon G\to \C^{\times}$, 
then clearly the Alexander polynomial $\Delta_G$ determines the 
characteristic varieties $\VV_d(G)$, for all $d\ge 1$, at least 
away from $1$.  

\section{Resonance varieties}
\label{sect:res vars}

We now look at the jumping loci associated to the 
cohomology ring of a space, and how they relate 
to the Alexander matrix and to the characteristic 
varieties.  

\subsection{Jump loci for the Aomoto complex}
\label{subsec:rv}
As before, let $X$ be a connected CW-complex with 
finite $k$-skeleton, for some $k\ge 1$.   Also, let $\k$ 
be a field; if $\ch\k =2$, assume additionally that 
$H_1(X,\Z)$ has no $2$-torsion.  

Consider the cohomology algebra $A=H^* (X,\k)$, with 
graded ranks the $\k$-Betti numbers, $b_i= \dim_{\k} A^i$. 
For each $a\in A^1$, we have $a^2=0$, by graded-commutativity 
of the cup product (and our assumption on the $2$-torsion). 
Thus, right-multiplication by $a$ defines a 
cochain complex 
\begin{equation}
\label{eq:aomoto}
\xymatrix{(A , \cdot a)\colon  \ 
A^0\ar^(.66){a}[r] & A^1\ar^{a}[r] & A^2  \ar[r]& \cdots},
\end{equation}
known as the {\em Aomoto complex}.  Let 
$\beta_i(A,a) = \dim_{\k} H^i(A,\cdot a)$ be the Betti 
numbers of this complex.   The jump loci for the 
Aomoto-Betti numbers define a natural filtration 
of the affine space $A^1=H^1(X,\k)$. 

\begin{definition}
\label{def:rvs}
The {\em resonance varieties} of $X$ (over $\k$) are the 
algebraic sets 
\begin{equation*}
\label{eq:rvs}
\RR^i_d(X,\k)=\{a \in A^1 \mid 
\beta_i(A,a) \ge  d\},  
\end{equation*}
defined for all integers $0\le i\le k$ and $d>0$. 
\end{definition}

It is readily seen that each of these sets is a 
homogeneous algebraic subvariety of $A^1=\k^{b_1}$.  
Indeed, $\beta_i(A,x a) =\beta_i(A,a)$, for all $x\in \k^{\times}$, 
and homogeneity follows. In each degree $i\ge 0$, the 
resonance varieties provide a descending filtration,
\begin{equation}
\label{eq:res filt}
H^1(X,\k) \supseteq \RR^i_1(X,\k) \supseteq \cdots 
\supseteq \RR^i_{b_i}(X,\k) \supseteq \RR^i_{b_{i}+1}(X,\k)=\emptyset.
\end{equation} 

Note that, if $A^i=0$, then $\RR^i_d(X,\k)=\emptyset$, for all $d>0$. 
In degree $0$, we have $\RR^0_1(X, \k)= \{ 0\}$,
and $\RR^0_d(X, \k)= \emptyset$, for $d>1$. In degree $1$, 
the varieties $\RR^1_d(X,\k)$ depend only on the group 
$G=\pi_1(X)$---in fact, only on the cup-product map 
$\cup \colon H^1(G,\k) \wedge H^1(G,\k) \to H^2(G,\k)$---%
so we sometimes denote them by $\RR_d(G,\k)$. 

The resonance varieties depend only on the 
characteristic of the ground field: 
if $\k\subseteq \mathbb{K}$ is an extension, then
$\RR^i_d(X,\k)=\RR^i_d(X,\mathbb{K}) \cap H^1(X,\k)$. 
Moreover, 
\begin{equation}
\label{eq:res prod}
\RR^i_1(X_1\times X_2,\k)= 
\bigcup_{p+q=i} \RR^{p}_1(X_1,\k) \times \RR^{q}_1(X_2,\k), 
\end{equation}
provided both $X_1$ and $X_2$ have finitely many 
$i$-cells, see \cite[Proposition 13.1]{PS-bns}. 

\subsection{Matrix interpretation}
\label{subsec:lin alex}

An alternate way to compute the degree $1$ 
resonance varieties of the algebra $A=H^*(X,\k)$ 
is to realize them as the determinantal varieties 
of a certain matrix of linear forms. Let us describe 
this method, following the approach taken in \cite{MS00}. 

By definition, an element $a\in A^1$ belongs to $R_d(X,\k)$ if 
and only if there exists a subspace $W\subset A^1$ of 
dimension $d+1$ such that $a\cup b=0$, for all $b\in W$.  
Fix ordered bases, $\{\alpha_1,\dots ,\alpha_{b_1}\}$ for 
$A^1$ and  $\{\beta_1,\dots , \beta_{b_2}\}$ for $A^2$.  
The multiplication map,  
$\mu\colon A^1 \otimes A^1 \to A^2$, is then given by 
\begin{equation}
\label{eq:mult}
\mu(\alpha_i , \alpha_j) =
\sum_{k=1}^{b_2}\mu_{ijk}\,  \beta_k,
\end{equation}
with coefficients $\mu_{ijk}\in \k$ satisfying 
$\mu_{jik}=-\mu_{ijk}$. Denote by $A_1$ the dual $\k$-vector 
space to $A^1$, and identify the symmetric algebra 
$S=\Sym(A_1)$ with the polynomial ring $\C[x_1,\dots ,x_{b_1}]$, 
where $x_i$ is the dual of $\alpha_i$.  With this notation, define 
the {\em linearized Alexander matrix}\/ of $A$ as  
\begin{equation}
\label{eq:linalex}
\Theta_A=\big( \Theta_{kj} \big) \colon S^{b_2} \to S^{b_1}, \quad 
\text{where $\Theta_{kj}=\sum_{i=1}^{b_1} \mu_{ijk} x_i$}. 
\end{equation}

Adapting the proof of \cite[Theorem 3.9]{MS00} to 
this slightly more general context, we obtain the following 
result. 

\begin{prop}[\cite{MS00}]
\label{prop:res linalex}
With notation as above, 
\begin{equation}
\label{eq:res linalex}
\RR_{d}(X,\k)=V(E_{d}(\Theta_A)).
\end{equation}
\end{prop}

Now let $G$ be a finitely presented group. Set $H=\ab(G)$, and 
identify $\Lambda=\k{H}$ with $\k[t_1^{\pm 1}, \dots , t_n^{\pm 1}]$, 
where $n=\rank H$. Let $I=I_H$ be the augmentation ideal. 
The $I$-adic completion, $\widehat{\Lambda}$, may be 
identified with the power series ring $\k[[t_1, \dots , t_n]]$, 
while the associated graded ring, $\gr(\widehat{\Lambda})$, 
may be identified with the polynomial ring $S=\k[x_1,\dots, x_n]$, 
where $x_i=t_1-1$.  

The next result (adapted from \cite[\S3.5]{MS00}) describes 
the relationship between the Alexander matrix 
from \eqref{eq:alex mat} and 
the linearized Alexander matrix from \eqref{eq:linalex}, 
thereby justifying {\it a posteriori}\/ the terminology. 

\begin{prop}[\cite{MS00}]
\label{prop:linalex}
Let $G=\langle x_1,\dots ,x_{q}\mid r_1,\dots ,r_m\rangle$
be a finitely presented group.  
Let $\Phi_G\colon  \Lambda^m \to \Lambda^q$ be the 
Alexander matrix, and 
$\gr(\widehat{\Phi}_G) \colon S^{m} \to S^{q}$ 
 its image under the $\gr$ functor.
Finally, let $A=H^*(X,\k)$ be the cohomology ring 
of the presentation $2$-complex.   Then:
\begin{equation}
\label{eq:gr alex}
\gr(\widehat{\Phi}_G)=\Theta_A.
\end{equation}
\end{prop} 

\subsection{Tangent cone inclusion}
\label{subsec:tc inclusion}
The above discussion hints at a relationship between the 
characteristic varieties $\VV_d(X,\k)$---the determinantal 
varieties of the Alexander matrix, at least away from $1$---and 
the resonance varieties $\RR_d(X,\k)$---the determinantal 
varieties of the linearized Alexander matrix.  This relationship, 
explored in the context of hyperplane arrangements in \cite{CS99}, 
was established in full generality by Libgober in \cite{Li02},  
at least for $\k=\C$. 

\begin{theorem}[\cite{Li02}]
\label{thm:lib}
Let $X$ be a connected CW-complex with finite $k$-skeleton.  
Then, for all $i<k$ and all $d>0$, 
\begin{equation}
\label{eq:lib}
TC_1(\VV_d^i(X, \C))\subseteq \RR_d^i(X, \C).
\end{equation}
\end{theorem}

For many spaces $X$, equality holds in \eqref{eq:lib}.  
We illustrate this phenomenon with the class of 
spaces discussed in \S\ref{subsec:toric complexes}. 

\begin{example}
\label{ex:tc toric}
Let $L$ be a simplicial complex on finite vertex set $\sV$, and 
let $X=T_L$ be the associated toric complex.  As shown in 
\cite[Theorem 3.8]{PS-toric}, the resonance varieties 
$\RR_d^i(T_L, \k)$ are given by the exact same expression 
as in \eqref{eq:cv toric}, with the subtorus $(\k^{\times})^{\sW}$ 
replaced by the coordinate subspace $\k^{\sW}$.  It follows that 
the exponential map $\exp\colon \C^{\sV}\to (\C^{\times})^{\sV}$ 
restricts to an isomorphism of analytic germs, 
$\exp\colon (\RR^i_d(T_L),0) \isom (\VV^i_d(T_L),1)$, 
for all $i\ge 0$. In particular, 
\begin{equation}
\label{eq:tcone toric}
TC_1(\VV_d^i(T_L))= \RR_d^i(T_L).
\end{equation}
\end{example}

On the other hand, as first noted in \cite[Remark 10.3]{MS02}, 
the inclusion from Theorem \ref{thm:lib} can be strict. We illustrate 
this point with a much simpler example than the original one. 

\begin{example}
\label{ex:tc heis}
Let $M=G_{\R}/G_{\Z}$ be the $3$-dimensional Heisenberg 
nilmanifold, where $G_{\R}$ is the group of real, 
unipotent $3\times 3$ matrices, and $G_{\Z}=\pi_1(M)$ 
is the subgroup of integral matrices in $G_{\R}$.   
A straightforward computation shows that $\VV_1(M)=\{1\}$, 
and thus $TC_1(\VV_1(M))=\{0\}$.  On the other hand, 
$\RR_1(M)=\C^2$, since the cup product vanishes on $H^1(M,\C)$. 
Thus,
\begin{equation}
\label{eq:tc heis}
TC_1(\VV_1(M))\varsubsetneqq \RR_1(M).
\end{equation}
\end{example}

For more information on the characteristic and resonance 
varieties of finitely generated nilpotent groups, we refer to  
M\u{a}cinic and Papadima \cite{MP} and M\u{a}cinic \cite{Mac}. 

\section{Bieri--Neumann--Strebel--Renz invariants}
\label{sect:bnsr}

In this section, we review the definition of the $\Sigma$-invariants 
of a group $G$ (and, more generally, of a space $X$), and discuss 
the relation of these invariants to the homology jumping loci. 

\subsection{$\Sigma$-invariants}
\label{subsec:bns}

We start with a definition given by  Bieri, Neumann, 
and Strebel in \cite{BNS}.   Let $G$ be a finitely generated 
group.  Choose a finite set of generators for $G$, and let 
$\CC(G)$ be the corresponding Cayley graph.  Given a 
homomorphism $\chi\colon G\to \R$, let $\CC_{\chi}(G)$ 
be the induced subgraph on vertex set 
$G_{\chi}=\{ g \in G \mid \chi(g)\ge 0\}$.  
The {\em BNS invariant}\/ of $G$ is the set   
\begin{equation}
\label{eq:bns inv}
\Sigma^1(G) =\{ \chi \in \Hom(G,\R) \setminus \{0\} \mid 
\text{$\mathcal{C}_{\chi}(G)$ is connected} \}.
\end{equation}
As shown in \cite{BNS}, $\Sigma^1(G)$ is an {\em open}, 
conical subset of the vector space $\Hom(G,\R)=H^1(G,\R)$; 
moreover, this set is independent of the choice of 
generators for $G$.

In \cite{BR}, Bieri and Renz extended this definition, as follows. 
Recall that a group $G$ (or, more generally, a monoid $G$) 
is of type $\FP_k$ if there is a projective $\Z{G}$-resolution 
$P_{\bullet}\to \Z$, with $P_i$ finitely generated, for all $i\le k$. 
In particular, $G$ is of type $\FP_1$ if and only if $G$ is finitely 
generated. The {\em BNSR invariants}\/ of $G$ are the sets 
\begin{equation}
\label{eq:bnsr def}
\Sigma^q(G,\Z)=\set{
\chi\in \Hom (G, \R)\setminus\{0\} \mid 
\text{the monoid $G_{\chi}$ is of type $\FP_q$}}.
\end{equation}
These sets form a descending chain of open 
subsets of $\Hom(G,\R)$, starting at $\Sigma^1(G,\Z)=\Sigma^1(G)$.  
Moreover, $\Sigma^q(G,\Z)$ is non-empty only if 
$G$ is of type $\FP_q$.   

To a large extent, the importance of the $\Sigma$-invariants 
lies in the fact that they control the finiteness properties 
of kernels of projections to abelian quotients.  More precisely, 
let $N$ be a normal subgroup of $G$, with $G/N$ abelian. 
Then, as shown in \cite{BNS, BR}, the group $N$ is of type $\FP_q$ 
if and only if  $ \{\chi \in \Hom (G,\R) \setminus \{0\} \mid \chi(N)=0\} 
\subseteq \Sigma^q(G, \Z)$. 
In particular, the kernel of an epimorphism $\chi\colon G \surj \Z$ is 
finitely generated if and only if both $\chi$ and $-\chi$ belong 
to $\Sigma^1(G)$.

\subsection{Novikov homology}
\label{subsec:novikov}
In \cite{Si}, Sikorav reinterpreted the BNS invariant of a 
finitely generated group $G$ in terms of Novikov homology. 
This interpretation was extended to all BNSR invariants 
by Bieri \cite{Bi07}, leading to a very general definition of the 
BNSR invariants of a space \cite{FGS, PS-bns}. 

The {\em Novikov--Sikorav completion}\/ of the group ring $\Z{G}$ 
with respect to a homomorphism $\chi\colon G\to \R$ consists of 
all formal sums $\lambda =\sum_i n_i g_i$, with $n_i\in \Z$ and 
$g_i\in G$, having the property that, for each $c\in \R$, the set 
of indices $ i$ for which $n_i \ne 0 $ and $\chi(g_i) \ge c$ is finite. 
With the obvious addition and multiplication, the Novikov--Sikorav 
completion, $\widehat{\Z{G}}_{\chi}$, is a ring, containing $\Z{G}$ 
as a subring; in particular,  $\widehat{\Z{G}}_{\chi}$ carries a natural 
$G$-module structure.  For details, we refer to Farber's book \cite{Far}. 

Now let $X$ be a connected CW-complex with finite $1$-skeleton, 
and let $G=\pi_1(X,x_0)$ be its fundamental group.  
The {\em BNSR invariants}\/ of $X$ are the sets 
\begin{equation}
\label{eq:bnsr fgs}
\Sigma^q(X, \Z)= \{ \chi\in \Hom(G, \R)\setminus \set{0} \mid
H_{i}(X, \widehat{\Z{G}}_{-\chi})=0, \ \forall\, i\le q \}.
\end{equation}
Identifying $\Hom(G,\R)=H^1(G,\R)=H^1(X,\R)$, we may view 
the $\Sigma$-invariants of $X$ as subsets of the vector space 
$H^1(X,\R)$.  As shown in \cite{Bi07}, definitions \eqref{eq:bnsr def} 
and \eqref{eq:bnsr fgs} agree: If $G$ is a group of type $\FP_k$, 
then $\Sigma^q(G, \Z)= \Sigma^q(K(G, 1), \Z)$, for all $q\le k$.   

\subsection{$\Sigma$-invariants and characteristic varieties}
\label{subsec:bns bound}
In practice, the BNSR invariants are extremely hard to compute: 
a complete description of the sets $\Sigma^q(G, \Z)$ is known only 
for some very special classes of groups $G$, such as one-relator 
groups and right-angled Artin groups.  The following result from 
\cite{PS-bns} gives a computable ``upper bound" for the 
$\Sigma$-invariants of a space $X$ or a group $G$, in 
terms of the real points on the exponential tangent cones 
to the respective characteristic varieties (discussed 
in \S\ref{subsec:exp tcone}). 

\begin{theorem}[\cite{PS-bns}]
\label{thm:bns tau}
Let $X$ be a connected CW-complex with finite $k$-skeleton, 
for some $k\ge 1$.  Then, for each $q\le k$, the following holds:
\begin{equation}
\label{eq:bns bound1}
\Sigma^q(X, \Z)\subseteq \Big( H^1(X,\R)\cap \bigcup_{i\le q} 
\tau_1\big(\VV_1^i(X, \C)\big)\Big)^{\compl}.
\end{equation}
In particular, for every finitely generated group $G$, 
\begin{equation}
\label{eq:bns bound2}
\Sigma^1(G)\subseteq 
\big( H^1(G,\R) \cap \tau_1 ( \VV_1(G,\C))\big)^{\compl}.
\end{equation}
\end{theorem}

Qualitatively, the above theorem says that each $\Sigma$-invariant 
is contained in the complement of a union of rationally defined 
subspaces.   As noted in \cite{PS-bns},  bound \eqref{eq:bns bound1}  
is sharp. For example, if $G$ is a finitely generated nilpotent group, 
then $\Sigma^q(G, \Z) = \Hom(G,\R)\setminus\{0\}$, while 
$\VV^q_1(G,\C)=\{1\}$, for all $q\ge 1$.  Thus, \eqref{eq:bns bound1} 
holds as an equality for $X=K(G,1)$. 

\subsection{$\Sigma$-invariants and valuations}
\label{subsec:valuations}

In \cite{De08}, Delzant discovered a surprising connection 
between the BNS invariant $\Sigma^1(G)$, discrete 
valuations on a field $\k$, and the first characteristic variety 
$\VV^1_1(G,\k)$.  Delzant's result was extended in \cite{PS-bns}, 
as follows. 

\begin{theorem}[\cite{PS-bns}]
\label{thm:bns valuation}
With notation as above, let $\rho\colon G\to \k^{\times}$ be 
a homomorphism  such that $\rho \in \bigcup_{i\le q} \VV^i_1(X, \k)$, 
for some $q\le k$, and let $v\colon \k^{\times}\to \R$ be a 
valuation such that $v\circ \rho\ne 0$. Then 
$v\circ \rho \not\in \Sigma^q(X, \Z)$. 

Conversely, let $\chi\colon G\to \R$ be a 
homomorphism such that $-\chi \in \Sigma^q(X, \Z)$, for 
some $q\le k$, and let  $\xi\colon G\surj \Gamma$ be  
the corestriction of $\chi$ to its image.  Suppose there is 
a character $\rho\colon  \Gamma \to \k^{\times}$ which is 
not an algebraic integer.  Then 
$\rho\circ \xi \not\in \bigcup_{i\le q} \VV^i_1(X, \k)$.
\end{theorem}

Here, $\rho$ is an {\em algebraic integer}\/ if there is an 
element $\Delta=\sum n_{\gamma}\gamma \in \Z\Gamma$ 
with $\Delta(\rho)=0$ and $n_{\gamma_0}=1$, 
where $\gamma_0$ is the greatest element 
of $\supp (\Delta)$.

\section{Formality properties}
\label{sect:formal}

In this section, all spaces $X$ will be 
connected, and homotopy equivalent to a CW-complex 
with finite $1$-skeleton. Likewise, all groups $G$ will be 
assumed to be finitely generated. 

\subsection{CDGAs and formality}
\label{subsec:formal dga}
Fix a ground field $\k$, of characteristic $0$.  
A commutative differential graded algebra (for short, 
a CDGA), is a graded, graded-commutative $\k$-algebra $A$, 
endowed with a differential $d_A\colon A\to A$ of degree $1$. 
A CDGA morphism is a quasi-isomorphism if it induces 
an isomorphism in cohomology.  Two CDGAs, $A$ and $B$, 
are said to be weakly equivalent if there is a zig-zag of 
quasi-isomorphisms (going both ways), connecting $A$ to $B$. 

A CDGA is {\em formal}\/ if it is weakly equivalent to its 
cohomology algebra, endowed with the zero differential. 
A CDGA $(A,d_A)$ is {\em $q$-formal}, for some $q\ge 1$, 
if there is a zig-zag of morphisms connecting $(A,d_A)$ to 
$(H^*(A, d_A), d=0)$, with each one of these maps inducing 
an isomorphism in cohomology up to degree $q$, and a 
monomorphism in degree $q+1$.  

The best-known formality test is provided by the (higher-order)
Massey products.  Let us briefly recall the definition of these 
cohomology operations in the simplest case. Suppose 
$\alpha_1,\alpha_2,\alpha_3$ are homogeneous elements 
in $H^*(A)$ such that $\alpha_1 \alpha_2 =\alpha_2\alpha_3=0$.  
Pick representative cocycles $a_i$ for $\alpha_i$, 
and elements $x, y\in A$ such that $dx=a_1a_2$ 
and $dy=a_2a_3$. It is readily checked that 
$xa_3-(-1)^{\abs{a_1}} a_1y$ is a cocyle.  
The Massey triple product
$\angl{\alpha_1,\alpha_2,\alpha_3}$ is 
the set of cohomology classes of all 
such cocycles.   The image of this set in the 
quotient ring $H^*(A)/( \alpha_1,\alpha_3)$ is a 
well-defined element of degree 
$|\alpha_1|+|\alpha_2|+|\alpha_3|-1$. The triple 
product $\angl{\alpha_1,\alpha_2,\alpha_3}$ 
is {\em non-vanishing}\/ if this element does not equal 
$0$.  If $(A,d)$ is formal, then all Massey products 
of order $3$ (or higher) in $H^*(A)$ vanish.  We refer to 
\cite{DGMS, Su77} for proofs and further details. 

\subsection{Formality of spaces}
\label{subsec:formal spaces}
Given a space $X$ as above, Sullivan \cite{Su77}
constructs an algebra $A_{\PL}(X)$ of polynomial differential 
forms on $X$ with coefficients in $\k$, and provides it with 
a natural CDGA structure.  The space $X$ is said to be formal 
(over $\k$) if Sullivan's algebra $A_{\PL}(X)$ is formal;   
likewise, $X$ is $q$-formal if this CDGA is $q$-formal.  
When $X$ is a smooth manifold, and $\k=\R$ or $\C$, 
one may replace in the definition the algebra of polynomial 
forms by de~Rham's algebra of differential forms.  

Examples of formal spaces include rational cohomology 
tori, surfaces, toric complexes, compact connected Lie 
groups, as well as their classifying spaces.   On the 
other hand, the only nilmanifolds which are formal 
are tori. Formality is preserved under wedges and 
products of spaces, and connected sums of manifolds.  

Whether a space $X$ is $1$-formal or not depends only 
on its fundamental group, $G=\pi_1(X,x_0)$. Indeed, 
if $f\colon X\to K(G,1)$ is a classifying map, then the 
induced homomorphism, $f^*\colon H^i(G,\k) \to H^i(X,\k)$, 
is an isomorphism for $i=1$ and a monomorphism for $i=2$. 

\subsection{$1$-formality of groups}
\label{subsec:1-formal}
In \cite{Qu}, Quillen associates to any group $G$ 
a pronilpotent, filtered Lie algebra, called 
the {\em Malcev completion}\/ of $G$. Concretely, 
the group ring $\k{G}$ has a Hopf algebra structure, 
with comultiplication given by $g \mapsto g\otimes g$, 
and with counit the augmentation map.  This Hopf algebra 
structure naturally extends to the completion of $\k{G}$ with 
respect to powers of the augmentation ideal.  The Malcev 
completion of $G$, denoted $\m(G)$, is the Lie algebra of 
primitive elements in $\widehat{\k{G}}$, equipped with the 
inverse limit filtration.

For a finitely generated group $G$, the $1$-formality property 
is equivalent to the quadraticity of $\m(G)$.  More precisely, 
$G$ is $1$-formal if and only if $\m(G)$ is isomorphic, 
as a filtered Lie algebra, to the completion with respect 
to degree of a quadratic Lie algebra.  

Examples of $1$-formal groups include finitely generated free 
groups and free abelian groups (more generally, right-angled 
Artin groups), surface groups, and groups with first Betti 
number equal to $0$ or $1$.  The $1$-formality property is 
preserved under free products and direct products.  

\subsection{The tangent cone theorem}
\label{subsec:tcone}

The main bridge between the $1$-formality property of a group 
and the nature of its cohomology jumping loci is provided by the 
following result, which summarizes Theorems A and B from 
\cite{DPS-duke} in the present setting. 

\begin{theorem}[\cite{DPS-duke}]
\label{thm:tcone}
Let $G$ be a $1$-formal group.  For each $d>0$,
\begin{romenum}

\item \label{tc1} The exponential map 
$\exp\colon H^1(G, \C) \to H^1(G, \C^{\times})$ restricts 
to an isomorphism of analytic germs, 
$\exp\colon (\RR_d(G),0)  \isom(\VV_d(G),1)$. 

\item \label{tc2} The following ``tangent cone formula" holds:
$\tau_1(\VV_d(G))=TC_1(\VV_d(G))=\RR_d(G)$. 

\item \label{tc3} The irreducible components of $\RR_d(G)$ 
are all linear subspaces, defined over $\Q$.

\item \label{tc4} 
The components of $\VV_d(G)$ passing through 
the origin are all rational subtori of the form $\exp(L)$, with $L$ 
running through the irreducible components of $\RR_d(G)$. 

\end{romenum}
\end{theorem}

The tangent cone formula from part \eqref{tc2} of the 
theorem can be employed as a non-formality test. We illustrate 
how this works with a well-known example of a manifold, 
whose lack of formality is usually detected by means of 
triple Massey products. 

\begin{example}
\label{ex:heisenberg}
Let $M=G_{\R}/G_{\Z}$ be the Heisenberg nilmanifold. 
As we saw in Example \ref{ex:tc heis}, 
$TC_1(\VV_1(G_{\Z}))\ne \RR_1(G_{\Z})$.  Thus, $G_{\Z}$ 
is not $1$-formal, and hence $M$ is not formal. 
\end{example}

Next, we illustrate how the rationality property from 
part \eqref{tc3} can also be used to detect non-formality. 

\begin{example}[\cite{DPS-duke}]
\label{ex:sqr2}
Consider the group 
$G=\langle x_1, x_2, x_3, x_4 \mid r_1, r_2, r_3\rangle$, 
with relators $r_1=[x_1, x_2]$, $r_2=[x_1, x_4] [x_2^{-2}, x_3]$, 
and $r_3= [x_1^{-1}, x_3] [x_2, x_4]$.  
Then $\RR_1(G)=\{z \in \C^4 \mid z_1^2-2z_2^2=0\}$ 
splits into linear subspaces defined over $\R$, but not 
over $\Q$.  Thus, $G$ is {\em not}\/ $1$-formal. 
\end{example}

\subsection{Resonance upper bound}
\label{subsec:res bound}
We return now to the BNSR invariants.  As before, let $X$ 
be a connected CW-complex with finite $k$-skeleton, for some 
$k\ge 1$.  Suppose there is an integer $q\le k$ such that 
the exponential map induces an isomorphism of analytic 
germs, $\exp\colon (\RR^i_d(G),0)  \isom(\VV^i_d(G),1)$, 
for all $i\le q$. We then get from Theorem \ref{thm:bns tau} 
the following ``resonance upper bound" for the $q$-th 
$\Sigma$-invariant of $X$:
\begin{equation}
\label{eq:bnsr res bound}
\Sigma^q(X, \Z)\subseteq \Big( \bigcup_{i\le q} 
\RR_1^i(X, \R) \Big)^{\compl}.
\end{equation}

\begin{example}
\label{eq:toric rb}
For a toric complex $T_L$ associated to a finite simplicial 
complex $L$, the discussion from Example \ref{ex:tc toric}, 
yields $\Sigma^q(T_L, \Z)\subseteq \big( \bigcup_{i\le q} 
\RR_1^i(T_L, \R) \big)^{\compl}$, for all $q\ge 1$, where 
\begin{equation}
\label{eq:res toric}
\RR^i_d(T_L, \R)=  \bigcup_{\stackrel{\sW \subseteq \sV}{%
\sum_{\sigma\in L_{\sV\setminus \sW}}
\dim_{\k} \widetilde{H}_{i-1-\abs{\sigma}} 
(\lk_{L_\sW}(\sigma),\k) \ge d}} \R^{\sW}. 
\end{equation}
For a right-angled group $G_{\Gamma}$ associated 
to a graph $\Gamma$, the BNSR invariants were 
computed by Meier, Meinert, and Van-Vyck in \cite{MMV}. 
Comparing these invariants with the resonance varieties 
of $G_{\Gamma}$ as in \cite[\S 14.3]{PS-bns} reveals 
that the equality 
\begin{equation}
\label{eq:bnsr raag}
\Sigma^q(G_\Gamma, \Z)= \big( \bigcup_{i\le q} 
\RR_1^i(G_\Gamma, \R) \big)^{\compl} 
\end{equation}
holds, provided the following condition is satisfied: 
for every simplex $\sigma$ of $\Delta=\Delta_{\G}$, 
and for every subset $\sW\subseteq \sV$ such that 
$\sigma \cap W=\emptyset$, the homology groups   
$\widetilde{H}_{j} ( \lk_{\Delta_{\sW}}(\sigma),\Z)$ 
are torsion-free, for all $j\le q-\dim(\sigma)-2$.  
This condition is satisfied, for instance, when 
$\Gamma$ is a tree, or $q=1$. 
\end{example}

As an immediate application of Theorem \ref{thm:tcone}, we have 
the following corollary.  

\begin{corollary}[\cite{PS-bns}]
\label{cor:formal bns bound}
Let $G$ be a $1$-formal group.  Then 
\begin{equation}
\label{eq:bns res bound}
\Sigma^1(G) \subseteq H^1(G, \R)\setminus \RR_1(G,\R).
\end{equation}
\end{corollary}

Recall that $\RR_1(G,\R)$ is a  homogeneous subvariety of 
$H^1(G, \R)$.  Thus, for equality to hold in \eqref{eq:bns res bound}, 
the set $\Sigma^1(G)$ must be symmetric about the origin, i.e., 
$\Sigma^1(G) =-\Sigma^1(G)$. This does not happen, in general. 

\begin{example}
\label{ex:bs group}
Let $G=\langle x_1,x_2\mid x_1^{\,}x_2x_1^{-1}=x_2^2\rangle$. 
Then $H^1(G,\R)=\R$; thus, $G$ is $1$-formal, and 
$\RR_1(G,\R)=\{0\}$.  On the other hand, it follows 
from \cite[Theorem 7.3]{BR} that $\Sigma^1(G)=(-\infty, 0)$; 
in particular, $\Sigma^1(G) \ne -\Sigma^1(G)$ and 
$\Sigma^1(G) \ne \RR_1(G,\R)^{\compl}$. 
\end{example}

In Example \ref{ex:link2}, we will exhibit a $1$-formal group 
$G$ for which $\Sigma^1(G)=-\Sigma^1(G)$, and yet 
$\Sigma^1(G) \ne \RR_1(G,\R)^{\compl}$.

Finally, let us note that  inclusion \eqref{eq:bns res bound} 
may fail to hold when $G$ is not $1$-formal:  
If $M=G_{\R}/G_{\Z}$ is the Heisenberg nilmanifold, then 
$\Sigma^1(G)=H^1(G,\R)\setminus \{0\}$, yet 
$\RR_1(G,\R)^{\compl}=\emptyset$. 

\section{K\"{a}hler and quasi-K\"{a}hler manifolds}
\label{sect:kahler}

In this section, we discuss the Alexander polynomial, 
cohomology jumping loci, and BNSR invariants 
of K\"{a}hler and quasi-K\"{a}hler manifolds and their 
fundamental groups. 

\subsection{K\"{a}hler manifolds and formality}
\label{subsec:kahler}

A compact, connected, complex manifold $M$ is called 
a {\em K\"{a}hler manifold}\/ if $M$ admits a Hermitian 
metric $h$ for which the imaginary part $\omega=\Im(h)$ 
is a closed $2$-form.   The best known examples are 
smooth, complex projective varieties. 

A finitely presented group $G$ is said to be a {\em K\"{a}hler group}\/  
if it can be realized as $G=\pi_1(M)$, where $M$ is a compact 
K\"{a}hler manifold. If $M$ can be chosen to be a smooth, irreducible, 
complex projective variety, then $G$ is said to be a {\em projective 
group}.  Both classes of groups are closed under finite direct products. 
Clearly, every projective group is a K\"{a}hler group, but 
whether the converse holds is an open problem. 

For example, $G=\Z^{2r}$ is the fundamental group 
of the complex torus $(S^1\times S^1)^r$, and thus a 
projective group.  Furthermore, all finite groups are 
projective, by a classical result of J.-P.~Serre.  For 
background information  on K\"{a}hler groups, we 
refer to the monograph \cite{ABCKT}.

If $M$ is a compact K\"{a}hler manifold, then each cohomology 
group $H^m(M,\Z)$ admits a pure Hodge structure of weight $m$.   
That is, the vector space $H^m(M,\Z)\otimes \C$ decomposes 
as a direct sum $\bigoplus_{p+q=m} H^{p,q}$, with $H^{q,p}$
the complex conjugate of $H^{p,q}$. 
Hodge theory has strong implications on the topology 
of compact K\"{a}hler manifolds.  For example, 
their odd Betti numbers must be even. Consequently, 
if $G$ is a K\"{a}hler group, then $b_1(G)$ must be even. 

A deeper constraint was established by Deligne, Griffiths, 
Morgan, and Sullivan in  \cite{DGMS}.   For a compact 
K\"{a}hler manifold $M$, let $d$ be the exterior derivative, 
$J$ the complex structure, and $d^c=J^{-1}dJ$. Then the 
following holds: If $\alpha$ is a form which is closed for 
both $d$ and $d^c$, and exact for either $d$ or $d^c$, 
then $\alpha$ is exact  for $dd^c$.  As a consequence 
of this ``$dd^c$ Lemma," all compact K\"{a}hler manifolds 
are formal; in particular, all K\"{a}hler groups are $1$-formal. 

\subsection{Quasi-K\"{a}hler manifolds}
\label{subsec:quasi kahler}

A manifold $X$ is said to be a {\em quasi-K\"{a}hler manifold}\/ 
if there is a compact K\"{a}hler manifold $\overline{X}$ and a 
normal-crossings divisor $D$ such that $X=\overline{X}\setminus D$. 
Smooth, irreducible, quasi-projective complex varieties 
are examples of quasi-K\"{a}hler manifolds.   

The notions of quasi-K\"{a}hler group and quasi-projective 
group are defined as above.   Both classes of groups 
are closed under finite direct products.  Clearly, every 
quasi-projective group is a quasi-K\"{a}hler group, but 
again, whether the converse holds is an open problem. 
Furthermore, every K\"{a}hler group is a quasi-K\"{a}hler 
group, but the converse does not hold; for instance, 
$\Z=\pi_1(\C^{\times})$ is a quasi-projective, non-K\"{a}hler group.

By a well-known result of Deligne~\cite{De2, De3}, each cohomology 
group $H=H^k(X,\Z)$ of a quasi-projective variety $X$ admits a 
mixed Hodge structure, that is, an increasing filtration $W_{\bullet}$ 
on $H_{\Q}=H\otimes \Q$, called the weight filtration, and a 
decreasing filtration $F^{\bullet}$ on $H_{\C}=H\otimes \C$, 
called the Hodge filtration, such that, for each $m\ge 0$, 
the associated graded piece $\gr^{W}_m(H_{\Q})$, together 
with the filtration induced by $F^{\bullet}$ on $H_{\C}$, is 
a pure Hodge structure of weight $m$.  Similarly, a 
quasi-K\"{a}hler manifold $X$ inherits a mixed Hodge 
structure from each compactification $\overline{X}$ 
as above; if $X$ is a smooth, quasi-projective variety, 
this structure is unique.

For a quasi-K\"{a}hler manifold $X$, the existence of mixed Hodge 
structures on its cohomology groups puts definite constraints on the 
topology of $X$.  Not surprisingly, these constraints are weaker than 
in the K\"{a}hler case.  For instance, quasi-K\"{a}hler groups need 
not be $1$-formal.  As an illustration, let $X$ be complex Heisenberg 
manifold, i.e., the total space of the $\C^{\times}$-bundle over 
$\C^{\times}\times \C^{\times}$, with Euler number $1$; then $X$ 
is a smooth, quasi-projective variety which fails to be $1$-formal. 

In this framework, Morgan \cite{M} proved the following result: 
If $X$ is a smooth, quasi-projective variety with $W_1(H^1(X,\C))=0$, 
then $X$ is $1$-formal.  As noted by Deligne \cite{De2, De3}, $W_1$ 
vanishes whenever $X$ admits a non-singular compactification 
$\overline{X}$ with $b_1(\overline{X})=0$.  This happens, 
for instance, when $X$ is the complement of a hypersurface 
in $\CP^n$.  It follows that fundamental groups of complements 
of projective hypersurfaces are $1$-formal.

\begin{remark}
\label{rem:curves}
Let $C$ be an algebraic curve in $\CP^2$, with complement 
$X=\CP^2\setminus C$.  By the above, the group $\pi_1(X)$ 
is $1$-formal. In fact, as shown by M\u{a}cinic \cite{Mac} and 
Cogolludo--Matei \cite{CoM} (using different methods), a stronger 
conclusion holds in this case: the space $X$ is formal.
\end{remark}

\subsection{Characteristic varieties}
\label{subsec:cv kahler}

Foundational results on the structure of the cohomology 
support loci for local systems on smooth projective varieties, 
and more generally, on compact K\"{a}hler manifolds were 
obtained by  Beauville \cite{Be}, Green--Lazarsfeld \cite{GL}, 
Simpson \cite{Sp92}, and Campana  \cite{Cm01}.  A more 
general result, valid in the quasi-K\"{a}hler case, was obtained 
by Arapura \cite{Ar}.  
 
\begin{theorem}[\cite{Ar}]
\label{thm:arapura}
Let $X=\overline{X}\setminus D$ be a quasi-K\"{a}hler manifold.  
Then:
\begin{romenum}
\item   \label{a1} 
Each component of $\VV^1_1(X)$ is either an isolated unitary 
character, or of the form $\rho\cdot  f^*(H^1 (C, \C^{\times}))$, 
for some torsion character $\rho$ and some admissible map 
$f\colon X\to C$. 

\item  \label{a2} 
If either $X=\overline{X}$ or $b_1(\overline{X})=0$, 
then, for all $i\ge 0$ and $d\ge 1$, each component of 
$\VV^i_d(X)$ is of the form $\rho\cdot  f^*(H^1 (T, \C^{\times}))$, 
for some unitary character $\rho$ and some holomorphic  
map $f\colon X\to T$ to a complex torus. 

\end{romenum}
\end{theorem}

Here, a map $f\colon X \to C$  is said to be {\em admissible}\/ if 
$f$ is a holomorphic, surjective map to a connected, 
smooth complex curve $C$, and $f$ has a holomorphic, 
surjective extension with connected fibers to smooth 
compactifications, $\overline{f}\colon \overline{X} \to \overline{C}$, 
obtained by adding divisors with normal crossings; in particular, 
the generic fiber of $f$ is connected, and the induced 
homomorphism, $f_{\sharp}\colon \pi_1(X)\to \pi_1(C)$, is onto. 
A {\em complex torus}\/ is a complex Lie group $T$ which 
decomposes as a product of factors of the form $\C^{\times}$ or 
$S^1\times S^1$. 

For smooth, quasi-projective varieties $X$, the isolated points 
in $\VV^1_1(X)$ are actually torsion characters, see \cite{Bu}, \cite{ACM}.

As noted in \cite{PS-bns}, Arapura's theorem has the 
following immediate corollary.

\begin{corollary}
\label{thm:corollary}
For a quasi-K\"{a}hler manifold $X$, 
all the components of $\VV^i_d(X)$ passing through 
the origin of $\Hom(\pi_1(X),\C^{\times})$ are subtori, 
provided one of the following conditions holds. 
\begin{romenum}
\item \label{ar1}
$i=d=1$. 
\item  \label{ar2}
$X$ is K\"{a}hler.
\item  \label{ar3}
$W_1(H^1(X,\C))=0$.
\end{romenum}
\end{corollary}

In particular, $\tau_1(\VV^i_d(X))=TC_1(\VV^i_d(X))$, 
whenever one of the above conditions is satisfied. 

In \cite[Theorem~1.1]{Li09}, Libgober proves a ``local" version 
of Arapura's theorem, in which all the translations are done 
by characters of finite order. 

\begin{theorem}[\cite{Li09}]
\label{thm:lib mm}
Let $\mathcal{X}$ be a germ of a complex space with an isolated, 
normal singularity whose link is simply-connected, and let $\mathcal{D}$ 
be a divisor on $\mathcal{X}$ with $n$ irreducible components. 
Then:
\begin{romenum}
\item \label{li1}
The character group 
$\Hom(\pi_1(\mathcal{X}\setminus \mathcal{D}),\C^{\times})$ 
is the complex algebraic torus $(\C^{\times})^n$. 
\item \label{li2}
Each characteristic variety 
$\VV^i_d(\mathcal{X}\setminus \mathcal{D})$ is a finite union of 
complex algebraic subtori, possibly translated by roots of unity. 
\end{romenum}
\end{theorem} 

\subsection{Resonance varieties}
\label{subsec:res kahler}

As shown in Theorem C and Corollary 7.4 of \cite{DPS-duke}, 
the presence of a K\"{a}hler metric on a compact, connected, 
complex manifold $M$ imposes very stringent conditions on 
the degree~$1$ resonance varieties of $M$.  Likewise, the 
existence of a quasi-K\"{a}hler structure on an open manifold 
$X$ puts subtle geometric constraints on $\RR_d(X)$, 
provided $X$ is $1$-formal. These results---which use in an 
essential way Theorems \ref{thm:tcone} and \ref{thm:arapura}, 
as well as theorems from \cite{DGMS} and \cite{M} mentioned 
in \S\ref{subsec:kahler} and \S\ref{subsec:quasi kahler}---% 
may be summarized as follows. 

\begin{theorem}[\cite{DPS-duke}]
\label{thm:res kahler} 
Let $X$ be a quasi-K\"{a}hler manifold, with fundamental 
group $G=\pi_1(X)$, and let $\{ L_{\alpha}\}_{\alpha}$ 
be the collection of positive-dimensional, irreducible 
components of $\RR_1(G)$.  If $G$ is $1$-formal, then
\begin{romenum}
\item  \label{rk1}  
Each $L_{\alpha}$ is a $p$-isotropic linear subspace of 
$H^1(G, \C)$, of dimension at least $2p+2$, for some 
$p=p(\alpha) \in \{0,1\}$. 

\item  \label{rk2} 
If $\alpha \ne \beta $, then $L_{\alpha} \cap L_{\beta}=\{0\}$.

\item  \label{rk3} 
$\RR_d(G)=\{0\} \cup \bigcup\nolimits_\alpha L_{\alpha}$, 
where the union is over all $\alpha$ for which 
$\dim L_{\alpha}>d+p(\alpha)$.  

\end{romenum}
Furthermore, 
\begin{romenum}
\setcounter{enumi}{3}
\item  \label{rk4} If $X$ is a compact K\"{a}hler manifold, 
then $G$ is $1$-formal, and each $L_{\alpha}$ is even-dimensional 
and $1$-isotropic.

\item  \label{rk5} If $X$ is a smooth, quasi-projective variety, 
and $W_1(H^1(X,\C))=0$, then $G$ is $1$-formal, 
and each $L_{\alpha}$ is $0$-isotropic. 
\end{romenum}
\end{theorem}

Here, we say that a non-zero subspace $U\subseteq  H^1(G,\C)$ 
is {\em $p$-isotropic}\/ with respect to the cup-product 
map $\cup_G\colon H^1(G,\C)\wedge H^1(G,\C) \to H^2(G,\C)$ 
if the restriction of $\cup_G$ to $U\wedge U$ has rank $p$.  
For example, if $C$ is a smooth complex curve with $\chi(C)<0$, 
then $\RR_1(\pi_1(C),\C)=H^1(C,\C)$, and  $H^1(C,\C)$ is either 
$1$- or $0$-isotropic, according to whether $C$ is compact or not.

\subsection{Examples and applications}
\label{subsec:apps}
Theorem \ref{thm:res kahler} can be used in a variety of ways 
to derive information on the fundamental groups of quasi-projective 
varieties---or rule out certain groups from being quasi-projective. 
As a first application, we obtain the following corollary, 
by comparing the conclusions of parts \eqref{rk4} and \eqref{rk5}. 

\begin{corollary}
\label{cor:qkk}
Let $X$ be a smooth, quasi-projective variety with $W_1(H^1(X,\C))=0$.  
Let $G=\pi_1(X)$, and suppose $\RR_1(G)\ne \{0\}$.  Then $G$ is not 
a K\"{a}hler group (though $G$ is $1$-formal).
\end{corollary}

Using now Deligne's result mentioned in \S\ref{subsec:quasi kahler}, 
we obtain a further corollary. 

\begin{corollary}
\label{cor:hyper kahler}
Let $X$ be the complement of a hypersurface in $\CP^n$, and 
let $G=\pi_1(X)$. If $\RR_1(G)\ne \{0\}$, then $G$ is not 
a K\"{a}hler group.
\end{corollary}

The assumption $\RR_1(G)\ne \{0\}$ is really necessary.  
For example, take $X=\C^2 \setminus \{z_1z_2= 0\}$.  
Then $G=\Z^2$ is clearly a K\"{a}hler group, but 
 $\RR_1(G)= \{0\}$.

The linearity property from Theorem \ref{thm:res kahler}\eqref{rk1} 
only uses the $1$-formality assumption on the group $G$, and 
follows at once from Theorem \ref{thm:tcone}\eqref{tc3}.  This 
property can be used to show that certain quasi-projective groups 
are not $1$-formal. 

\begin{example}[\cite{DPS-duke}]
\label{ex:conf spaces}
Let $X=F(\Sigma_g,n)$ be the configuration space of $n$ 
labeled points on a Riemann surface of genus $g$.  Clearly,  
$X$ is a connected, smooth, quasi-projective variety.
Its fundamental group, $\pi_1(X)=P_{g,n}$, is the pure 
braid group on $n$ strings on $\Sigma_g$.  
The cohomology ring $H^*(F(\Sigma_g,n),\C)$ was 
computed by Totaro in \cite{Tot}.  Using this computation, 
we get
\begin{equation*} 
\label{eq:resg1}
\RR_1(P_{1,n})=\left\{ (x,y) \in \C^n\times \C^n \left|
\begin{array}{l}
\sum_{i=1}^n x_i=\sum_{i=1}^n y_i=0 ,\\[2pt]
x_i y_j-x_j y_i=0,  \text{ for $1\le i<j< n$}
\end{array}
\right\}. \right.
\end{equation*}
For $n\ge 3$, this is an irreducible, non-linear variety 
(a rational normal scroll). Hence, the group 
$P_{1,n}$ is not $1$-formal. 
\end{example}

The more refined isotropicity properties  from 
Theorem \ref{thm:res kahler}, 
parts \eqref{rk1}, \eqref{rk4}, and \eqref{rk5} 
use both the $1$-formality and the (quasi-) K\"{a}hlerianity 
assumptions on the group $G$.  These isotropicity 
properties of the components of $\RR_1(G)$ are utilized  
in \cite[Theorem 11.7]{DPS-duke} to achieve a 
complete classification of (quasi-) K\"{a}hler 
groups within the class of right-angled Artin groups 
(which, recall, are always $1$-formal). 

\begin{theorem}[\cite{DPS-duke}]
\label{thm:artinserre}
Let $\G$ be a finite simple graph, and $G_{\G}$  the 
corresponding right-angled Artin group.  Then:
\begin{enumerate}
\item \label{ak1} 
$G_{\G}$ is a quasi-K\"{a}hler group if and only if 
$\G$ is a complete multipartite graph $K_{n_1,\dots,n_r}=
\overline{K}_{n_1} * \cdots * \overline{K}_{n_r}$, in which 
case $G_{\G}= F_{n_1}\times \cdots \times F_{n_r}$.

\item  \label{ak2}
$G_{\G}$ is a quasi-K\"{a}hler group if and only if 
$\G$ is a complete graph $K_{2m}$, 
in which case $G_{\G}= \Z^{2m}$. 
\end{enumerate}
\end{theorem}

\subsection{Alexander polynomial}
\label{subsec:delta qk}

The approach we have been using so far in this section also 
informs on the Alexander polynomial of a (quasi-) K\"{a}hler 
group $G$.  The following theorem was proved in \cite{DPS-imrn}, 
assuming $G$ is (quasi-) projective; the same proof works 
in the stated generality. 

\begin{theorem}[\cite{DPS-imrn}]
\label{thm:alex qk}
Let $G$ be a quasi-K\"{a}hler group. Set $n=b_1(G)$, 
and let $\Delta_G$ be the Alexander polynomial of $G$. 
\begin{romenum}
\item \label{alk1}
If $n\ne 2$, then the Newton polytope of $\Delta_G$ 
is a line segment. 

\item \label{alk2}
If $G$ is actually a K\"{a}hler group, then $\Delta_G \doteq \const$.  
\end{romenum}
\end{theorem}

If $n\ge 3$,  we may write 
$\Delta_G (t_1,\dots , t_n)\doteq c P(t_1^{e_1}\cdots t_n^{e_n})$, 
for some $c\in \Z$, some polynomial $P\in \Z[t]$ equal to a product of 
cyclotomic polynomials, and some exponents $e_i\ge 1$ with 
$\gcd(e_1,\dots ,e_n)=1$. 

\subsection{Bieri--Neumann--Strebel invariants}
\label{subsec:bns kahler}

Recently, Delzant \cite{De1} found a very precise connection 
between the BNS invariant of a compact K\"{a}hler manifold $M$ 
and admissible maps $f\colon M\to C$. (Recall that such maps, 
also known as {\em pencils}, are holomorphic, surjective maps 
to smooth complex curves, and have connected generic fiber.) 

\begin{theorem}[Delzant \cite{De1}] 
\label{thm:delzant}
Let $M$ be a compact K\"{a}hler manifold, with $G=\pi_1(M)$. 
Then $\Sigma^1(G)^{\compl}= \bigcup f_{\alpha}^* 
\big( H^1(C_{\alpha}, \R)\big)$, where the union is taken 
over those pencils $f_{\alpha}\colon M\to C_{\alpha}$
with the property that either $\chi(C_{\alpha})<0$, or 
$\chi(C_{\alpha})=0$ and $f_{\alpha}$ has some multiple fiber. 
\end{theorem}

Recall from Corollary \ref{cor:formal bns bound} that 
the BNS invariant of a $1$-formal group $G$ is contained 
in the complement in $H^1(G, \R)$ of the resonance variety 
$\RR_1(G,\R)$.  In general, this inclusion is strict.  
Nevertheless the class 
of K\"{a}hler groups for which the aforementioned 
inclusion is an equality can be identified precisely. 
This is done in \cite[Theorem 16.4]{PS-bns}, 
using the above result of Delzant, together with 
work of Arapura \cite{Ar} and Theorem \ref{thm:res kahler}.

\begin{theorem}[\cite{PS-bns}]
\label{thm:kahler}
Let  $M$ be a compact K\"{a}hler manifold with $b_1(M)>0$, and 
let $G=\pi_1(M)$.  Then $\Sigma^1(G)=\RR_1(G, \R)^{\compl}$ 
if and only if there is no pencil $f\colon M\to E$ onto an elliptic 
curve $E$ such that $f$ has multiple fibers.  
\end{theorem}

The equality $\Sigma^1(G)=\RR_1(G, \R)^{\compl}$ does not 
hold for arbitrary K\"{a}hler groups $G$. The following example, 
based on a well-known construction of Beauville \cite{Be}, 
illustrates this point.

\begin{example}
\label{rem:beauville}
Let $Y$ be a compact K\"{a}hler manifold on which a finite 
group $\pi$ acts freely.  Let $E$ be an elliptic curve, and let 
$p\colon C \to E$ be a ramified, regular $\pi$-cover, 
with at least one ramification point. Clearly, the diagonal 
action of $\pi$ on $C \times Y$ is free; let 
$M=(C \times Y)/\pi$ be the orbit space. 
It is readily seen that $M$ is a compact K\"{a}hler 
manifold, with $b_1(M)>0$. 

Consider the commuting diagram
\begin{equation}
\label{eq:beauville}
\xymatrix{
C \times Y \ar^{q}[r] \ar^{\proj_1}[d] & M  \ar^{f}[d]  \\
C \ar^{p}[r] & E 
}
\end{equation}
where $q$ is the orbit map and $f$ is the map induced by 
the first-coordinate projection.  As noted in \cite[Example 1.8]{Be}, 
the fibers of $f$ over the ramification points of $p$ are 
multiple fibers, while the other (generic) fibers of $f$ are 
isomorphic to $Y$.  In other words, $f\colon M\to E$ is 
an elliptic pencil with multiple fibers.  Let $G=\pi_1(M)$.   
By Theorem \ref{thm:kahler}, then, the BNS invariant 
$\Sigma^1(G)$ is strictly contained in $\RR_1(G,\R)^{\compl}$.
\end{example}

\section{Hyperplane arrangements}
\label{sect:arrs}

\subsection{The complement of an arrangement}
\label{subsec:hyp}

A {\em hyperplane arrangement}\/ is a finite collection 
of hyperplanes in some complex affine space $\C^{\ell}$.  
The main topological object associated to an arrangement 
$\A$ is its {\em complement}, 
$X(\A)=\C^{\ell}\setminus\bigcup_{H\in \A}H$. 
This is a smooth, quasi-projective variety, whose topological 
invariants are intimately connected to the combinatorics of 
the arrangement, as encoded in the {\em intersection lattice}, 
$L(\A)$, which is the poset of all non-empty intersections of 
$\A$, ordered by reverse inclusion.  

\begin{example}
\label{ex:braid arr}
The best-known example is the braid arrangement 
$\A_\ell$, consisting of the diagonal hyperplanes in 
$\C^{\ell}$.  The complement is the configuration 
space $F(\C,\ell)$ of $\ell$ ordered points in $\C$, 
while the intersection lattice is the lattice of partitions 
of $[\ell]=\set{1,\dots,\ell}$, ordered by refinement.  
In the early 1960s,  Fox and Neuwirth showed that 
$\pi_1(X(\A_\ell))=P_{\ell}$, the pure braid group on 
$\ell$ strings, while Neuwirth and Fadell showed 
that $X(\A_\ell)$ is aspherical. 
\end{example}

For a general arrangement with complement $X=X(\A)$, 
the cohomology ring $H^*(X,\Z)$ was computed 
by Brieskorn in the early 1970s, building on pioneering 
work of Arnol'd on the cohomology ring of the braid arrangement.  
It follows from Brieskorn's work that the space $X$ is formal.
In 1980, Orlik and Solomon gave a simple combinatorial 
description of the ring $H^*(X,\Z)$:  it is the quotient 
$A=E/I$ of the exterior algebra $E$ on classes dual 
to the meridians, modulo a certain ideal $I$ determined 
by the intersection poset.   We refer to the book 
by Orlik and Terao~\cite{OT} for detailed explanations 
and further references.  

The fundamental group of the complement,  $G(\A)=\pi_1(X(\A))$, 
can be computed algorithmically, using the braid monodromy 
associated to a generic projection of a generic slice $\B$ in $\C^2$; 
see \cite{CS97} and references therein.  The end result is a 
finite presentation with generators $x_1,\dots ,x_n$ 
corresponding to the meridians (oriented compatibly with 
the complex orientations of  $\C^2$ and the lines in $\B$), 
and commutator relators of the form $x_i \alpha_j(x_i)^{-1}$, 
where $\alpha_j\in P_{n}$ are the (pure) braid monodromy 
generators, acting on the meridians via the Artin 
representation.  In particular, $H_1(G(\A),\Z)=\Z^n$, 
with preferred basis the images of the meridional generators. 

It should be noted that arrangement groups are not always 
combinatorially determined.  Indeed, G.~Rybnikov has 
produced a pair arrangements, $\A$ and $\A'$, with 
$L(\A)\cong L(\A')$, but $G(\A)\not\cong G(\A')$.  
For a detailed account of Rybnikov's celebrated 
example, we refer to \cite{ACCM}. 

\subsection{Resonance varieties of arrangements} 
\label{subsec:res arr}
Let $\A$ be a central arrangement in $\C^{\ell}$ 
(i.e., all hyperplanes of $\A$ pass through the origin). 
The resonance varieties $\RR^i_1(X(\A),\C)$ were 
first defined and studied by Falk in \cite{Fa97}.  The resonance 
varieties over an arbitrary field, $\RR^i_d(X(\A),\k)$, were  
considered by Matei and Suciu in \cite{MS00}, and  
investigated in detail by Falk in \cite{Fa07}.   The 
varieties $\RR^i_d(X(\A),\k)$ lie in the affine space 
$\k^n$, where $n=\abs{\A}$; they depend solely on 
the intersection lattice, $L(\A)$, and on the characteristic 
of the field $\k$.  A basic problem in the subject is to find 
concrete formulas making this dependence explicit. 

Best understood are the degree $1$ resonance varieties 
over the complex numbers, $\RR_d(\A)=\RR^1_d(X(\A),\C)$. 
These varieties admit a very precise combinatorial description, 
owing to work of Falk \cite{Fa97}, Cohen--Suciu \cite{CS99}, 
Libgober \cite{Li01}, Libgober--Yuzvinsky \cite{LY}, and others, 
with the state of the art being the recent work of 
Falk--Yuzvinsky \cite{FY}, Pereira--Yuzvinsky \cite{PeY}, 
and Yuzvinsky~\cite{Yu}.  Let us briefly describe these 
varieties, based on the original approach from \cite{Fa97}, 
with updates as warranted; for the latest approach, using 
``multinets," we refer to \cite{FY}. 

By the Lefschetz-type theorem of Hamm and L\^{e}, 
taking a generic two-dimen\-sional section does not 
change the fundamental group of the complement. 
Thus, in order to describe $\RR_1(\A)=\RR_1(G(\A))$, 
we may assume $\A=\{\ell_1,\dots ,\ell_n\}$ is an affine 
line arrangement in $\C^2$, for which no two lines are 
parallel.  The following facts are known:
\begin{enumerate}
\item The variety $\RR_1(\A)\subset \C^n$ lies in the 
hyperplane $\Delta_n=\{x\in \C^n \mid \sum_{i=1}^{n} x_i=0\}$.  
\item Each component is a linear subspace of 
dimension at least $2$. 
\item Two distinct components of $\RR_1(\A)$ meet 
only at $0$, and $\RR_d(\A)$ is the union of those 
subspaces of dimension greater than $d$.
\end{enumerate} 

\begin{example}
\label{ex:pencil}
Let $\A$ be a pencil of $n$ lines through the origin of 
$\C^2$, defined by the equation $z_1^n-z_2^n=0$. 
The fundamental group of the complement is 
$G=\langle x_1,\dots , x_n \mid 
\text{$x_1\cdots x_n$ central} \rangle$. 
If $n=1$ or $2$, then $\RR_1(\A)=\{0\}$; 
otherwise, $\RR_1(\A)=\cdots =\RR_{n-2}(\A)=\Delta_n$, 
and $\RR_{n-1}(\A)=\{0\}$. 
\end{example}

Returning to the general case, the simplest components of 
$\RR_1(\A)$ are the {\em local}\/ components:  to an intersection 
point $v_J=\bigcap_{j\in J} \ell_j$ of multiplicity $\abs{J} \ge 3$, 
there corresponds a subspace $L_J$ of dimension $\abs{J}-1$, given 
by equations of the form $\sum_{j\in J} x_j = 0$, and $x_i=0$ 
if $i\notin J$. 

The remaining components correspond to certain ``neighborly 
partitions" of sub-arrange\-ments of $\A$: to each such 
partition $\Pi$, there corresponds a subspace $L_{\Pi}$, given 
by equations of the form $\sum_{j\in \pi} x_j=0$, with $\pi$ 
running through the blocks of $\Pi$. Moreover, $\dim L_{\Pi}>0$ 
if and only if a certain bilinear form associated to $\Pi$ is degenerate.  

If $\abs{\A}\le 5$, then all components of $\RR_1(\A)$ are local.  
For $\abs{\A}\ge 6$, though, the resonance variety $\RR_1(\A)$ 
may have interesting components. 

\begin{example}
\label{ex:braid}  
Let $\A$ be a generic $3$-slice of the braid 
arrangement $\A_4$, with defining polynomial 
$Q(\A)=z_0z_1z_2(z_0-z_1)(z_0-z_2)(z_1-z_2)$. 
Take a generic plane section, and label the 
corresponding lines as $6,2,4,3,5,1$. 
The variety $\RR_{1}(\A)\subset \C^6$ has $4$ 
local components, corresponding to the triple 
points $124, 135, 236, 456$, 
and one non-local  component, corresponding to the 
neighborly partition $\Pi=(16| 25 | 34)$:
\[
\small{
\begin{aligned}
& L_{124}= \{ x_1 + x_2 + x_4=x_3=x_5=x_6=0 \} ,\ 
L_{135}=\{ x_1 + x_3 + x_5=x_2=x_4=x_6=0 \} ,  \\
& L_{236}= \{ x_2 + x_3 + x_6= x_1=x_4=x_5=0 \}, \
 L_{456}= \{ x_4 + x_5 + x_6=x_1=x_2=x_3=0 \} ,\\
& L_{\Pi}=\{ x_1+x_2+x_3=x_1- x_6=x_2-x_5=x_3-x_4=0 \} .
\end{aligned}
}
\]
Since all these components are $2$-dimensional, $\RR_2(\A)=\{0\}$. 
\end{example}

For an arbitrary arrangement $\A$, it follows from the work of 
Falk, Pereira, and Yuzvinsky \cite{FY, PeY, Yu} that any non-local 
component in $\RR_{1}(\A)$ has dimension either $2$ or $3$. 
The only known example for which non-local components of dimension 
$3$ occur is the Hessian arrangement of $12$ planes in $\C^3$, 
defined by the polynomial 
$Q(\A)=z_0z_1z_2\prod_{j=0}^{2} \prod_{k=0}^{2} 
(z_0+\omega^j z_1+\omega^k z_2)$,  where $\omega=\exp(2\pi i/3)$. 

\subsection{Characteristic varieties of arrangements} 
\label{subsec:cv arr}

Let $\A=\{H_1,\dots, H_n\}$ be a central arrangement in 
$\C^{\ell}$, and let $G(\A)=\pi_1(X(\A))$ be the fundamental group 
of its complement. From the discussion in \S\ref{subsec:hyp}, we 
know that $G(\A)$ is a quasi-projective, $1$-formal group.    
Let $\VV_d(\A)=\VV_d(G(\A),\C)$ be its (degree~$1$) characteristic 
varieties.  

From Arapura's Theorem \ref{thm:arapura}, we know that 
$\VV_d(\A)$ consists of subtori in $(\C^{\times})^n$, 
possibly translated by roots of unity (necessarily of 
finite order, if $d=1$), together with a finite number 
of isolated unitary characters.   By Theorem \ref{thm:tcone}, 
we have 
\begin{equation}
\label{eq:arr tcone}
TC_1(\VV_d(\A))=\RR_d(\A), \quad\text{for all $d\ge 1$}. 
\end{equation}
The tangent cone formula \eqref{eq:arr tcone} for complements 
of hyperplane arrangements was first proved (by different methods) 
by Cohen and Suciu \cite{CS99}, Libgober \cite{Li02}, and Libgober 
and Yuzvinsky \cite{LY}, and was generalized to the higher-degree 
jump loci by Cohen and Orlik \cite{CO00}. 

As an upshot of \eqref{eq:arr tcone}, the components of 
$\VV_1(\A)$ passing through the origin are completely determined 
by $\RR_1(\A)$, and thus, by the intersection lattice, $L(\A)$:  
to each (linear) component $L\subset \C^n$ of $\RR_1(\A)$ 
there corresponds a (torus) component $T=\exp(L) \subset 
(\C^{\times})^n$ of $\VV_1(\A)$.  The only novelty here is that 
some of these subtori may intersect away from the origin, 
sometimes in a character belonging to $\VV_2(\A)$---a 
phenomenon first noted in \cite[Example 4.4]{CS99}. 

\begin{example}
\label{ex:cv braid}  
For the arrangement $\A$ from Example \ref{ex:braid}, the 
variety $\VV_1(\A)$ has four local components---$V_{124}$, 
$V_{135}$, $V_{236}$, and $V_{456}$---and one non-local 
component, $V_{\Pi}$.  Any two of these components 
intersect only at $1$; moreover, $\VV_2(\A)=\{1\}$. 
Computing the Alexander matrix of $G(\A)$ and applying  
Proposition \ref{prop:alex comm} reveals that there are no 
translated components in $\VV_1(\A)$.
\end{example}

Not too much is known about the components of 
$\VV_1(\A)$ not passing through $1$.  A major 
open problem in the subject is to decide whether 
such ``translated" components are combinatorially 
determined.  That this is not a moot point is illustrated 
by the following example from \cite{Su02}.  

\begin{example}
\label{ex:deleted B3}
Let $\A$ be the deleted $\operatorname{B}_3$ arrangement, 
with defining polynomial 
$Q(\A)=z_0z_1(z_0^2-z_1^2)(z_0^2-z_2^2)(z_1^2-z_2^2)$. 
The variety $\VV_1(\A)$ contains $7$ local components, 
corresponding to $6$ triple points and one quadruple point, 
and $5$ non-local components passing through $1$, 
corresponding to braid sub-arrangements.  Additionally, 
$\VV_1(\A)$ contains a component of the form $\rho\cdot T$, 
where
\[
T=\set{(t^2,t^{-2},1,1,t^{-1},t^{-1},t,t) \mid t\in \C^{\times}}
\]
is a $1$-dimensional subtorus of $(\C^{\times})^8$, 
and $\rho=(1,1,-1,-1,-1,-1,1,1)$ is a root of unity of order $2$.  
\end{example}

In \cite{Di07}, Dimca provides an effective algorithm for 
detecting translated tori in the characteristic varieties 
of arrangements, such as the component $\rho\cdot T$ 
arising in the above example. In \cite{NR09}, Nazir and Raza 
prove the following result: If $\A$ has $2$ lines (or less) 
that contain all intersection points of multiplicity $3$ and 
higher, then $\VV_1(\A)$ has no translated components.  
For the deleted $\operatorname{B}_3$ arrangement, 
there exist precisely $3$ lines containing all high-multiplicity 
points, and so Nazir and Raza's result is best possible 
from this point of view. 

\subsection{Arrangements with parallel lines} 
\label{subsec:parallel}
The theory developed so far works with only minimal 
modifications for arbitrary affine line arrangements in $\C^2$, 
i.e, for arrangements $\A=\{\ell_1,\dots ,\ell_n\}$ which  
may contain parallel lines.  Homogenizing, we get 
an arrangement $\bar\A=\{L_0,L_1, \dots ,L_n\}$, 
with $L_0=\CP^2\setminus \C^2$ the line at infinity. 
Viewing $\bar\A$ as the projectivization of a central 
arrangement $\hat\A$ in $\C^3$, and taking a generic 
$2$-section $\B$ of $\hat\A$, we are back to the situation 
studied previously.

\begin{remark}
\label{rem:proj}
Note that the complement of $\A$ in $\C^2$ 
is the same as the complement of $\bar\A$ in $\CP^2$. 
Thus, the arrangement group, $G(\A)=\pi_1(X(\A))$, 
can be viewed as the fundamental group of the 
complement of a projective hypersurface.   
\end{remark}

 It is now an easy matter to relate 
the jump loci of $\A$ to those of $\B$.  For instance, the 
resonance variety $\RR_1(\A)\subset \C^n$ may be 
obtained from $\RR_1(\B)\subset \C^{n+1}$ by slicing 
with a suitable hyperplane.  

\begin{example}
\label{ex:parallel}
Let $\A$ be an arrangement of $n$ parallel lines in $\C^2$. 
Clearly,  $X(\A)=(\C\setminus \set{\text{$n$ points}})\times \C$, 
and thus $G(\A)=F_n$.  It is readily seen that $\B$ is a pencil 
of $n+1$ lines in $\C^2$, and thus, $G(\B)\cong \Z\times F_n$. 
This isomorphism identifies $\RR_1(\B)=\Delta_n$ with 
$\RR_1(\A)=\C^n$.
\end{example}

In general, to a family of $k\ge 2$ parallel lines in $\A$ there  
corresponds a pencil of $k+1$ lines in $\B$.  The corresponding 
$k$-dimensional local component of $\RR_1(\B)$ yields a 
$k$-dimensional component of $\RR_1(\A)$.  Similar 
considerations apply to non-local components. 

\begin{example}
\label{ex:braid decone}
Let $\A$ be the line arrangement defined by the 
polynomial $Q(\A)=z_1z_2(z_1-1)(z_2-1)(z_1-z_2)$.   
The corresponding central arrangement, $\hat\A$, 
is the arrangement from Example \ref{ex:braid}.   
Thus, $\RR_1(\A)$ has a (unique) non-local 
component, $L=\{ x_1+x_2+x_3=x_2-x_5=x_3-x_4=0 \}$. 
\end{example}

\subsection{Zariski-type theorems} 
\label{subsec:zariski}
A well-known theorem of Zariski asserts that an arrangement 
of lines in $\CP^2$ has only double points  if and only 
if the fundamental group of its complement is free abelian;  
see \cite[Theorem 1.1]{CDP} for details and references.  
One may wonder whether there are other classes of line 
arrangements which may be similarly characterized 
in terms of their groups. 

In \cite[Corollary 1.7]{CDP}, Choudary, Dimca, and Papadima 
prove an analogue of Zariski's theorem, in the setting of affine 
arrangements. Given an $r$-tuple of integers $m=(m_1,\dots,m_r)$, 
a line arrangement $\A$ in $\C^2$ is said to be of type $\A(m)$ 
if $\A$ consists of lines in $r$ parallel directions, each 
direction containing $m_i$ lines.  Clearly, 
such arrangements have only double points, and any nodal 
arrangement in $\C^2$ is of type $\A(m)$, for some $r$-tuple $m$. 

\begin{theorem}[\cite{CDP}]
\label{thm:cdp}
Let $\A$ be a line arrangement in $\C^2$, with 
group $G=G(\A)$. The following are equivalent:
\begin{romenum}
\item \label{cdp1}  $\A$ is of type $\A(m)$, for some 
$r$-tuple $m=(m_1,\dots,m_r)$.
\item \label{cdp2}  $G$ is isomorphic to 
$F_{m_1}\times \cdots \times F_{m_r}$, 
via an isomorphism preserving the standard 
generators in $H_1$. 
\end{romenum}
\end{theorem}

In particular, if a line arrangement $\A$ has only double 
points, then the fundamental group of its complement 
is isomorphic to a finite direct product of finitely generated 
free groups. 

The next result was inspired by recent work of Fan \cite{Fan09}.  
The implication $\eqref{p1} \Rightarrow \eqref{p4}$ below 
recovers the main result (Theorem 1) from \cite{Fan09}. 
It also improves on Theorem \ref{thm:cdp} in the particular case 
when $r=1$, by allowing arbitrary isomorphisms between 
$G$ and a free group.  We offer a completely different---and 
much shorter---proof of this implication, based on the techniques 
described so far.  

\begin{theorem}
\label{thm:fan}
Let $\A=\set{\ell_1,\dots, \ell_n}$ be an arrangement 
of lines in $\C^2$, with group $G=G(\A)$.  
The following are equivalent:
\begin{romenum}
\item \label{p1}  The group $G$ is a  free group.  
\item \label{p2}  The characteristic variety $\VV_1(\A)$ 
coincides with $(\C^{\times})^n$. 
\item \label{p3}  The resonance variety $\RR_1(\A)$ 
coincides with $\C^n$. 
\item \label{p4}  The lines $\ell_1,\dots, \ell_n$ are 
all parallel. 
\end{romenum}
\end{theorem}

\begin{proof}
\eqref{p1} $\Rightarrow$ \eqref{p2}. If $G$ is a free group 
(necessarily, of rank $n$), then, as noted in Example \ref{ex:wedge}, 
 $\VV_1(G)=(\C^{\times})^n$. 

\eqref{p2}  $\Rightarrow$ \eqref{p3}. If  $\VV_1(G)=(\C^{\times})^n$, 
then obviously $TC_1(\VV_1(G)) = \C^n$.  Hence, by formula \eqref{eq:lib} 
(or, alternatively, formula \eqref{eq:arr tcone}),  $\RR_1(G) = \C^n$. 

\eqref{p3} $\Rightarrow$ \eqref{p4}. Suppose $\A$ has a point 
of multiplicity $k\ge 3$.  Then $\RR_1(\A)$ has a (local) component 
of dimension $k-1$.  But $k\le n$, and so this contradicts 
$\RR_1(\A) = \C^n$. Thus, all multiple points of $\A$ must be  
double points, i.e, $\A$ must be of type $\A(m)$, for some 
$r$-tuple $m=(m_1,\dots,m_r)$. To each family of $m_i\ge 2$ 
parallel lines, there corresponds a component of $\RR_1(\A)$ 
of dimension $m_i$.  But $m_i<n$, unless $r=1$ and $m_1=n$, 
i.e., $\A$ consists of $n$ parallel lines. 

\eqref{p4} $\Rightarrow$ \eqref{p1}. If the lines are parallel, 
then, as noted in Example \ref{ex:parallel}, $G=F_n$.
\end{proof}

\subsection{Comparison with K\"{a}hler groups} 
\label{subsec:kahler arr}
Recall that arrangement groups are $1$-formal, quasi-projective 
groups.   But are they K\"{a}hler groups?  Of course, 
a necessary condition for $G=G(\A)$ to be a
K\"{a}hler group is that $b_1(G)=\abs{\A}$ must be 
even. In \cite[Example 3.48]{ABCKT}, it is noted that the 
group of a pencil of $n\ge 3$ lines cannot be realized as 
a ``non-fibered" K\"{a}hler group, i.e., the fundamental 
group of a compact K\"{a}hler manifold which does not 
fiber over a Riemann surface of genus $g\ge 2$. 

In the next theorem, we go much further, and identify 
precisely which arrangement groups are K\"{a}hler groups.

\begin{theorem}
\label{thm:arr kahler}
Let $\A$ be an arrangement of lines in $\C^2$, 
with group $G=G(\A)$. The following are equivalent:
\begin{romenum}
\item \label{h1} $G$ is a K\"{a}hler group. 
\item \label{h2} $G$ is a free abelian group of even rank.
\item \label{h3} $\A$ is an arrangement of an even number of lines 
in general position. 
\end{romenum}
\end{theorem}

\begin{proof}
Implication \eqref{h3} $\Rightarrow$ \eqref{h2} follows from  
Zariski's theorem, while  \eqref{h2} $\Rightarrow$ \eqref{h1}  
follows from the discussion in \S\ref{subsec:kahler}.  

To prove \eqref{h1} $\Rightarrow$ \eqref{h3}, assume 
$G=G(\A)$ is a K\"{a}hler group. If $\A$ is not in general 
position, i.e, $\A$  is not of type $\A(1,\dots,1)$, then 
either $\A$  has an intersection point of multiplicity $k+1\ge 3$, 
or $\A$ contains a family of $k\ge 2$ parallel lines.  In either 
case, $\RR_1(\A)$ has a $k$-dimensional component; in particular, 
$\RR_1(\A)\ne \set{0}$.  In view of Remark \ref{rem:proj}, 
the conclusion follows from Corollary \ref{cor:hyper kahler}. 
\end{proof}

\subsection{Comparison with right-angled Artin groups} 
\label{subsec:raag arr}
Arrangement groups also share many common features with 
right-angled Artin groups.  For instance, if $G$ is a group in 
either class, then $G$ is a commutator-relators group; 
$G$ is $1$-formal; and each resonance variety $\RR_d(G)$ 
is a union of linear subspaces.  Moreover, the free groups $F_n$ 
and the free abelian groups $\Z^n$ belong to both classes.  
So what exactly is the intersection of these two classes of 
naturally defined groups?  

To answer this question, we need one more definition, 
due to Fan \cite{Fan97}.   Given a line arrangement $\A$, 
its {\em multiplicity graph}, $\Gamma(\A)$, is the graph 
with vertices the intersection points with multiplicity at 
least $3$, and edges the segments between multiple 
points on lines which pass through more than one 
multiple point.

\begin{theorem}
\label{thm:arr raag}
Let $\A$ be an arrangement of lines in $\C^2$, 
with group $G=G(\A)$. The following are equivalent:
\begin{romenum}
\item \label{rg1} $G$ is a right-angled Artin group. 
\item \label{rg2} $G$ is a finite direct product of finitely 
generated free groups.
\item \label{rg3} The multiplicity graph $\Gamma(\A)$ is a forest. 
\end{romenum}
\end{theorem}

\begin{proof}
Recalling that every arrangement group is a quasi-K\"{a}hler group, 
equivalence \eqref{rg1} $\Leftrightarrow$ \eqref{rg2} follows 
at once from Theorem \ref{thm:artinserre}\eqref{ak1}.  

Implication \eqref{rg3} $\Rightarrow$ \eqref{rg2} 
is proved in \cite[Theorem 1]{Fan97}, while 
\eqref{rg2} $\Rightarrow$ \eqref{rg3} 
is proved in \cite[Theorem 6.1]{ELST}, if we keep  
in mind Remark 2.8 from that paper. 
\end{proof}

\subsection{Alexander polynomials of arrangements}
\label{subsec:alex hyp arr}
As we saw in Theorem \ref{thm:deltacd1}, there is a strong
relationship between the Alexander polynomial of a group $G$ 
and the codimension-$1$ strata of its degree-$1$ jump loci. 
This relationship was used in \cite[Proposition 3.4]{DPS-imrn} 
to compute the Alexander polynomial of $G=\pi_1(X(\A))$, 
in the case when $\A$ is a generic plane section of a central 
hyperplane arrangement.  A similar proof works for arbitrary 
affine line arrangements. For the sake of completeness, we 
go over the argument, with the required modifications. 

Let $\A$ be an arrangement of $n$ lines in $\C^2$, 
with group $G=G(\A)$.  Recall $H=\ab(G)$ is isomorphic 
to $\Z^n$, and comes equipped with a preferred basis,  
corresponding to the oriented meridians around the 
lines; this yields an identification of $\Z{H}$ with 
$\Lambda=\Z[t_1^{\pm 1}, \dots , t_n^{\pm 1}]$. 
Define the Alexander polynomial of the arrangement 
to be $\Delta_{\A}:=\Delta_G$, viewed 
as a Laurent polynomial in $\Lambda$.   This polynomial 
depends (up to normalization) only on the group $G$ and 
its peripheral structure; thus, only on the homeomorphism 
type of the complement $X(\A)$.   

\begin{example}
\label{ex:alex pen}
Let $\A$ a pencil of $n$ lines.  Recall we have 
$G=\langle x_1,\dots , x_n \mid \text{$x_1\cdots x_n$ central} \rangle$. 
It is readily checked that $\Delta_\A=0$ if $n=1$; 
$\Delta_\A=1$ if $n=2$; and 
$\Delta_\A= (t_1\cdots t_n-1)^{n-2}$, otherwise.  
\end{example}

\begin{example}
\label{ex:alex parallel}
Let $\A$ be an arrangement of $n-1$ parallel lines, 
together with a transverse line, i.e., an arrangement 
of type $\A(n-1,1)$. Then 
$G=\langle x_1,\dots , x_n \mid  \text{$x_n$ central} \rangle
=F_{n-1}\times \Z$.  
It is readily checked that $\Delta_\A=0$ if $n=1$; 
$\Delta_\A=1$ if $n=2$; and 
$\Delta_\A= (t_n-1)^{n-2}$, otherwise.  
\end{example}

Note that, for $n\ge 3$, the corresponding pair of arrangements 
have isomorphic groups, but non-homeomorphic complements: 
the difference is picked up by the Alexander polynomial. 

\begin{theorem} 
\label{thm:delta arr}
Let $\A$ be an arrangement of $n$ lines in $\C^2$, 
with Alexander polynomial $\Delta_\A$.  
\begin{romenum}
\item \label{aa1}
If $\A$ is a pencil and $n\ge 3$, then 
$\Delta_\A= (t_1\cdots t_n-1)^{n-2}$. 
\item \label{aa2}
If $\A$ is of type $\A(n-1,1)$ and $n\ge 3$, then 
$\Delta_\A= (t_n-1)^{n-2}$.
\item \label{aa3}
For all other arrangements, $\Delta_\A\doteq \const$. 
\end{romenum} 
\end{theorem}

\begin{proof}
Cases \eqref{aa1} and \eqref{aa2} have been treated 
already in the previous two examples.  Thus, we 
may assume $\A$ does not belong to either class; 
in particular, $n\ge 3$.  

From \S\S\ref{subsec:cv arr}--\ref{subsec:parallel}, 
we know that intersection points of multiplicity $k+1\ge 3$, 
and families of $k\ge 2$ parallel lines give rise to local 
components of $\VV_1(\A)$, of dimension $k$.  
In both situations, we must have $k\le n-2$,    
since we are in case \eqref{aa3}.  

If $n\le 5$, then all components of $\VV_1(\A)$ 
are local, except if $\A$ is the arrangement of $5$ lines 
from Example \ref{ex:braid decone}, in which case 
$\VV_1(\A)$ has a $2$-dimensional global component. 
If $n\ge 6$, the variety $\VV_1(\A)$ may very well have 
non-local components; yet, as shown in \cite[Theorem~7.2]{PeY}, 
the dimension of any such component is at most $4$. 

Putting things together, we see that  all components of 
$\VV_1(\A)$ must have codimension at least $2$; in 
other words, the codimension-$1$ stratum of this 
variety, $\cv(\A)$, is empty. Since $n\ge 2$,  
Theorem \ref{thm:deltacd1}\eqref{dc3} implies that 
$\Delta_\A\doteq \const$.  
\end{proof}

\subsection{$\Sigma$-invariants of arrangements} 
\label{subsec:bns arr}
Let $\A$ be a central arrangement in $\C^{\ell}$, 
with complement $X=X(\A)$.  It follows from the 
work of Esnault, Schechtman, and Viehweg \cite{ESV} 
that the exponential map, 
$\exp\colon H^1(X,\C) \to H^1(X,\C^{\times})$, 
induces an isomorphism of analytic germs, 
$\exp\colon (\RR^i_d(X,\C),0)  \isom(\VV^i_d(X,\C),1)$, 
for all $i\ge 0$ and all $d>0$. In particular, the resonance 
varieties $\RR^i_d(X,\C)$ are all finite unions of rationally 
defined linear subspaces.  Formula \eqref{eq:bnsr res bound}  
now yields the following combinatorial upper bound for the 
BNSR invariants of an arrangement. 

\begin{prop}[\cite{PS-bns}]
\label{prop:bns hyp}
Let $\A$ be a hyperplane arrangement in $\C^{\ell}$, 
with complement $X$. Then
$\Sigma^q(X,\Z)\subseteq \big(\bigcup_{i\le q} 
\RR^i_1(X, \R)\big)^{\compl}$, for all $q\ge 0$.  
\end{prop}

On the other hand, it is known from the work 
of Eisenbud, Popescu, and Yuzvinsky \cite{EPY} 
that resonance ``propagates" for the Orlik-Solomon 
algebra of an arrangement.  More precisely, if 
$i< j \le \ell$, then $\RR^{i}_1(X,\C) \subseteq \RR^{j}_1(X,\C)$. 
Using this fact, we get the following immediate corollary. 

\begin{corollary}
\label{cor:bns prop}
Let $X$ be a hyperplane arrangement complement. Then
$\Sigma^q(X,\Z)\subseteq \RR^q_1(X, \R)^{\compl}$, 
for all $q\ge 0$.  
\end{corollary}

The obvious question now is whether equality holds 
in the above corollary.  In the case when $q=1$, this is 
a question about the BNS invariant of the arrangement group.  
We state this question (as well as a couple of related 
ones), as follows. 

\begin{question}
\label{quest:bns arr}
Let $\A$ be an arrangement of lines in $\C^2$, with 
group $G=G(\A)$. 
\begin{romenum}
\item Does the equality $\Sigma^1 (G)=- \Sigma^1 (G)$ hold?  
\item Even stronger, does the equality 
$\Sigma^1(G)=\RR_1(G,\R)^{\compl}$ hold?
\item If it doesn't, is the BNS invariant combinatorially determined?
\end{romenum}
\end{question}
For a complexified real arrangement, complex conjugation in 
$\C^{2}$ restricts to a homeomorphism $\alpha\colon X \to X$ 
with the property that $\alpha_*= - \id$ on $H_1(X, \Z)$. 
It follows that $\Sigma^1 (G)=- \Sigma^1 (G)$, which is 
consistent with the symmetry property of $\RR_1(G, \R)^{\compl}$. 

\section{Milnor fibrations}
\label{sect:milnor}

\subsection{The Milnor fibration of a polynomial}
\label{subsec:mf poly}

Let $f\in \C[z_0,\dots, z_d]$ be a weighted homogeneous polynomial 
of degree $n$, with positive integer weights $(w_0,\dots, w_d)$ 
assigned to the variables.  Denote by $V=V(f)$ the zero-set 
of the polynomial function $f\colon \C^{d+1} \to \C$, and let 
$X:=\C^{d+1} \setminus V$ be the complement of this variety. 

As shown by Milnor \cite{Mi}, the restriction $f\colon X \to \C^{\times}$ 
is the projection map of a smooth, locally trivial bundle, nowadays known 
as the (global) {\em  Milnor fibration}\/  of $f$.  The {\it Milnor fiber}, 
$F=F(f)$, is simply the typical fiber of this fibration, $F:=f^{-1}(1)$.  
The Milnor fiber of $f$ is a smooth affine variety, having the homotopy 
type of a $d$-dimensional, finite CW-complex. We will assume 
throughout that $f$ is not a proper power, so that the 
Milnor fiber $F$ is connected.

The {\em geometric monodromy}\/ of the Milnor fibration is 
the map  $h\colon F\to F$,  $(z_0,\dots,z_d) \mapsto 
(\zeta_n^{w_0} z_0,\dots,\zeta_n^{w_d} z_d)$, 
where $\zeta_n = \exp(2\pi i/n)$.  Plainly, the mapping 
torus of $h$ is homeomorphic to the complement $X$.   
Moreover, the map $h$ generates a cyclic group $\Z_n$ 
which acts freely on $F$.  This free action gives rise to a 
regular $n$-fold covering, $F \to F/\Z_n$.

\subsection{The Milnor fiber as a finite cyclic cover}
\label{subsec:mf cyclic}
Let us describe the cover $F \to F/\Z_n$ more concretely, at least 
in the case when $f$ is homogeneous.  Following the approach 
from \cite{CS95}  (with a few more details added in for completeness), 
we build an interlocking set of fibrations in which this cover fits.  
These fibrations will be recorded in Diagram \eqref{eq:mf diagram} below. 

Start with the map $p\colon\C^{d+1}\setminus\{0\} \to \CP^d$, 
$z=(z_0,\dots,z_d) \mapsto (z_0:\dots:z_d)$; this is the 
projection map of the Hopf bundle, with fiber $\C^{\times}$.  
Denote by $U$ the complement of the hypersurface in $\CP^{d}$ 
defined by $f$.  The map $p$ restricts to a bundle map, 
$p_X\colon X\to U$, also with fiber $\C^{\times}$. 
By homogeneity, $f(wz)=w^n f(z)$, for every $z\in X$ and 
$w\in \C^{\times}$. Thus, the restriction of $f$ to a fiber of 
$p_X$ may be identified with the covering projection 
$q\colon \C^{\times}\to \C^{\times}$, $q(w)=w^n$.   

Let $E=\set{(z,w)\in X\times \C^{\times}\mid f(z)=w^n}$, and let 
$\pi\colon E\to X$ and $\phi\colon E\to \C^{\times}$ be the restrictions 
of the coordinate projections on $X\times \C^{\times}$.  Clearly, 
$f\circ \pi=q\circ \phi$.  Thus, $\pi$ is a regular $\Z_n$-cover, 
obtained as the pullback of $q$ along $f$.  
\begin{equation}
\label{eq:mf diagram}
\xymatrix{
\Z_n \ar[dr] & \Z_n \ar[dr]  \\
\Z_n \ar[dr] & E \ar[r]^{\phi} \ar[d]^{\psi} \ar[dr]^{\pi} 
& \C^{\times}\ar[d] \ar[dr]^{q} \\
& F\ar[r] \ar[dr]^{p_F} & X\ar[r]^{f}\ar[d]^{p_X} & \C^{\times}\\  
& & U 
}
\end{equation}

Now let $p_F\colon F \to U$ be the restriction 
of $p_X$ to the Milnor fiber. It is readily seen that this is the 
orbit map of the free action of the geometric monodromy on 
$F$; hence, $F/{\Z}_n \cong  U$. 

Finally, given an element $(z,w)\in E$, note that 
$f(w^{-1}z)=w^{-n}f(z)=1$.  Hence, we may define a map 
$\psi\colon E\to F$ by $\psi(z,w)=w^{-1}z$.  Plainly, 
$p_F\circ \psi = p_X\circ \pi$.  Therefore, the pullback of 
$p_F$ along $p_X$ coincides with the covering map 
$\pi\colon E\to X$.  

\subsection{Identifying the cyclic cover}
\label{subsec:char hom}
We now determine the classifying homomorphism 
$\lambda\colon \pi_1(U)\surj \Z_n$ of the regular, cyclic 
$n$-fold cover $p_{_F}\colon F\to U$.  

Let $f=f_1^{a_1} \cdots f_s^{a_s}$ be the factorization 
into distinct irreducible factors of $f$. Recall we assume 
$f$ is not a proper power, i.e., $\gcd(a_1,\dots , a_s)=1$. 
The hypersurface $V$ decomposes into irreducible 
components, $V_i=f_i^{-1}(0)$.  Choose as generators 
of $\pi_1(X)$ meridian circles $\gamma_i$ about $V_i$, 
with orientations compatible with the complex orientations 
of $\C^{d+1}$ and $V_i$.  

The map $f_{\sharp}\colon \pi_1(X)\to \Z$ factors through 
$f_*\colon H_1(X,\Z)\to \Z$, which is given by 
$f_*([\gamma_i])=a_i$, see \cite[pp.~76--77]{Di92}.  On the 
other hand, the map $(p_{_{X}})_{*}\colon H_1(X,\Z)\to H_1(U,\Z)$ 
may be identified with the canonical projection 
$\Z^s \surj \Z^s/(n_1,\dots,n_s)$, where 
$n_i=\deg (f_i)$, see \cite[pp.~102--103]{Di92}.

Putting these observations together, we obtain 
the following generalization of Proposition 1.2 
from \cite{CS95}.

\begin{prop}
\label{prop:mf cover} 
The classifying homomorphism of the covering 
$p_{_F}\colon F\to U$ is the composition  $\lambda=\alpha\circ \ab$, 
where $\ab\colon \pi_1(U) \rightarrow H_1(U) \cong
{\Z}^s/(n_1,\dots,n_s)$ is the abelianization map, and 
$\alpha\colon H_1(U) \rightarrow {\Z}_n$ is induced by the
composition $\Z^s \xrightarrow{ (a_1,\dots,a_s) }  \Z \surj \Z_n$. 
\end{prop}

Alternatively, if we identify  $\pi_1(U)=\pi_1(X) / 
\langle \gamma_1^{n_1}\cdots \gamma_s^{n_s}\rangle$  
and write $\Z_n=\langle g \mid g^n=1\rangle$, then   
the epimorphism $\lambda\colon \pi_1(U)\surj \Z_n$ 
is given by $\lambda(\gamma_i) = g^{a_i}$. 

\subsection{Homology of the Milnor fiber}
\label{subsec:mf homology}
In \cite{Mi}, Milnor determined the homotopy type of the 
Milnor fiber in an important special case: If $f=f(z_0,\dots, z_d)$ 
has an isolated critical point at $0$ of multiplicity $\mu$, then 
$F(f)\simeq \bigvee^{\mu} S^d$.  As shown by Milnor and Orlik, 
the multiplicity can be calculated in terms of the degree and the 
weights as $\mu=(n/w_0 -1)\cdots (n/w_d-1)$. 

For non-isolated singularities, though, no general formula for 
the homology of the Milnor fiber is known. Nevertheless, 
Theorem \ref{thm:betti1 cover}, together with 
Proposition \ref{prop:mf cover}, yields a formula for 
the homology groups $H_1(F,\k)$, in terms of the 
characteristic varieties $V_d(U,\k)$ of the projectivized 
complement, under suitable assumptions on the field $\k$. 

\begin{theorem}
\label{thm:b1 mf}
Let $f=f_1^{a_1} \cdots f_s^{a_s}$ be a homogeneous 
polynomial of degree $n$, with irreducible factors $f_i$, 
and with $\gcd(a_1,\dots , a_s)=1$.  
Let $\k$ be an algebraically closed field of characteristic 
not dividing $n$, and fix a primitive $n$-th root of 
unity $\zeta\in \k$. Then,
\begin{equation}
\label{eq:betti1 mf}
\dim_{\k} H_1(F, \k) =s-1 + \sum_{1\ne k | n} \varphi(k) 
\depth_{\k} \big(\rho^{n/k}\big), 
\end{equation}
where $\rho\colon \pi_1(U)\to \k^{\times}$ is the character 
defined on meridians by $\rho(\gamma_i) = \zeta^{a_i}$. 
\end{theorem}

In the case where $\deg(f_i)=1$ for $1\le i\le s$, this result 
reduces to Theorem 4.4 from \cite{CDS}. 

\subsection{Milnor fibers of arrangements}
\label{subsec:mf arr}

Much effort has been put over the past two decades in 
computing the homology groups of the Milnor fiber  
of a hyperplane arrangement. For our purposes here, 
we shall restrict to central arrangements in $\C^3$, 
defined by polynomials of the form $f=f_1\cdots f_n$, 
with $f_i=f_i(z_0, z_1, z_2)$ linear forms.  

Let $\A$ be such an arrangement, with 
complement $X$, projectivized complement $U$, 
and meridians  $\gamma_1,\dots ,\gamma_n$.  Since 
$\CP^{d}\setminus \{f_i=0\}\cong \C^d$, the bundle  
$p_X\colon X\to U$ is trivial.   This gives a diffeomorphism 
$X\cong U\times \C^{\times}$, under which the group 
$\pi_1(\C^{\times})$ corresponds to 
$\Z=\langle \gamma_1\cdots \gamma_n \rangle$.  
Thus, the characteristic varieties of $\pi_1(U)$ are given by
\begin{equation}
\label{eq:vux}
\VV_d(U,\k)=\{ t\in (\k^{\times})^n \mid t\in \text{$\VV_d(X,\k)$ 
and $t_1\cdots t_n=1$}\}. 
\end{equation}
In particular, $\VV_d(U,\C)\subset (\C^{\times})^n$ 
is the intersection of $\VV_d(\A)=\VV_d(X,\C)$ with 
the subtorus $(\C^{\times})^{n-1}=\set{t_1\cdots t_n=1}$.  

Some well-known questions arise.  

\begin{question}
\label{quest:mf}
Let $F=\{f=1\}$ be the Milnor fiber of a central arrangement 
$\A$ in $\C^3$. 
\begin{romenum}
\item  \label{qmf1}
Is $H_1(F,\Z)$ torsion-free?
\item \label{qmf2}
Is $H_1(F,\Z)$ determined by the intersection lattice $L(\A)$?
\item \label{qmf3}
Is $b_1(F)$ determined by $L(\A)$?
\end{romenum} 
\end{question}

Of course, an affirmative answer to \eqref{qmf2} implies one 
for \eqref{qmf3}, while an  affirmative answer to both \eqref{qmf1} 
and  \eqref{qmf3} implies one for \eqref{qmf2}.  In view of 
Theorem \ref{thm:b1 mf}, if one could show that all translated 
components of $\VV_d(U,\C)$ are combinatorially determined, 
one would obtain an affirmative answer to \eqref{qmf3}. 

\begin{example}
\label{ex:milnor braid}  
Let $\A$ be a generic $3$-slice of the braid arrangement, 
with defining polynomial $f=z_0z_1z_2(z_0-z_1)(z_0-z_2)(z_1-z_2)$.
Let $U$ be the projectivized complement of $\A$, and let 
$\rho\colon \pi_1(U)\to \C$, $\gamma_i \mapsto \zeta_6$ 
be the character corresponding to the Milnor fiber $F$.  
Recalling from Example \ref{ex:cv braid} the decomposition 
into irreducible components of $\VV_1(U)$, we see that 
$\lambda^2$ belongs to $V_{\Pi}$, yet $\lambda\notin \VV_1(U)$.  
Since $\VV_2(U)=\{1\}$, we find 
\[
b_1(F)=5 + \varphi(3)\cdot \depth_{\C} (\lambda^2) = 5+2\cdot 1=7. 
\]
Direct computation with the Jacobian matrix of $\pi_1(U)$, based 
on the method outlined in Proposition \ref{prop:fox} shows that, 
in fact, $H_1(F,\Z)=\Z^7$. 
\end{example}

If one is willing to go to multi-arrangements, then the answer 
to Question \ref{quest:mf}\eqref{qmf1} is negative, as the 
following example from \cite[\S 7.3]{CDS} illustrates. 

\begin{example} 
\label{ex:mf deleted B3}
Let $\A$ be the be the deleted $\operatorname{B}_3$ 
arrangement discussed in Example \ref{ex:deleted B3}.  
Assigning weights $(2,1,3,3,2,2,1,1)$ to the 
factors of $Q(\A)$, we obtain a new homogeneous 
polynomial (of degree $15$), 
\[
f=z_0^2z_1(z_0^2-z_1^2)^3(z_0^2-z_2^2)^2(z_1^2-z_2^2).
\]

Note that $\A$ and $V(f)$ share the same (projectivized) 
complement, yet the Milnor fibers of $Q(\A)$ and $f$ are 
different.  Applying Theorem \ref{thm:b1 mf} to the Milnor fiber 
$F=F(f)$, we find that $\dim_\k H_1(F, \k)=7$ if $\ch\k \ne 2,3$, 
or $5$, yet $\dim_\k H_1(F, \k)=9$ if $\ch\k =2$.  

Direct computation with the Jacobian matrix of $\pi_1(U)$ 
yields the precise formula for the homology with integer 
coefficients:  
\[
H_1(F, \Z)=\Z^7\oplus\Z_2\oplus\Z_2.
\]
\end{example}

\subsection{Formality of the Milnor fiber}
\label{subsec:mf formality}

The following question was raised in \cite{PS-formal}:  Is the 
Milnor fiber of a reduced polynomial, $F(f)$, always $1$-formal? 
This question has led to a certain amount of recent activity, out 
of which two completely different examples came out, showing 
that, in general, the Milnor fiber is {\em not}\/ formal. 

\begin{example}[Zuber \cite{Zu}]
\label{ex:zuber}
Let $\A=\A(3,3,3)$ be the monomial arrangement in $\C^3$ 
defined by the polynomial 
\[
f(z_0,z_1,z_2)=(z_0^3-z_1^3)(z_0^3-z_2^3)(z_1^3-z_2^3),
\]
and let $F$ be the Milnor fiber of $f$.  It is shown in \cite{Zu} 
that $TC_1(\VV_1(F,C))\ne \RR_1(F,\C)$; hence, by 
Theorem \ref{thm:tcone}, $F$ is not $1$-formal. 
\end{example}

\begin{example}[Fern\'andez de Bobadilla \cite{FdB}]
\label{ex:bobadilla}
Consider the polynomial 
\[
f(z_0,\dots , z_{10}) = z_0z_2z_3z_5z_6 + z_0z_2z_4 z_7+ 
z_1z_2z_4z_8 + z_1z_3z_5z_9 + z_1z_3z_4z_{10}.
\]
Then $f$ is weighted homogeneous of degree $5$, with weights 
$(1,1,1,1,1,1,1,2,2,2,2)$. As shown in \cite{FdB}, the Milnor fiber 
$F$ is homotopy equivalent to the complement of the coordinate
subspace arrangement $\A=\A_K$ from \cite[\S 9.3]{DeS}.  
This arrangement consists of subspaces $H_1,\dots , H_5$ 
in $\C^6$, given by $H_{i}=\set{x_i=x_{i+1}=0}$.  

It is easy to see that $F$ is $2$-connected.  Computations 
from \cite[\S 6.2]{DeS}  show that  there are classes 
$\alpha, \beta, \gamma\in H^3(F,\Z)=\Z^5$ such that 
the triple Massey product 
$\langle \alpha, \beta, \gamma\rangle \in H^8(F,\Z)=\Z$ 
is defined, with $0$ indeterminacy, and does not vanish.   
Hence, $F$ is not formal. 
\end{example}

In view of these examples, Question 5.5 from \cite{PS-formal} 
needs to be reformulated, as follows.
\begin{question}
\label{question:formal mf}
For which (weighted) homogeneous polynomials $f$ is 
the Milnor fiber $F(f)$ a $1$-formal (or even formal) space?
\end{question}

For hyperplane arrangements, we can ask an even more precise question. 
\begin{question}
\label{quest:formal mf arrr}
Let $\A$ be a central arrangement in $\C^3$, with Milnor fiber $F$.
\begin{romenum}
\item Is the Milnor fiber $1$-formal? 
\item If not, can this be detected by means of (triple) Massey 
products in $H^2(F,\Q)$?
\item Is the formality property of the Milnor fiber 
combinatorially determined? 
\end{romenum}
\end{question}

\section{Three-dimensional manifolds}
\label{sect:3mfd}

\subsection{Alexander polynomial and characteristic varieties}
\label{subsec:res3}

Throughout this section, $M$ will be a compact, connected 
$3$-manifold, possibly with boundary $\partial M$.  Let 
$G=\pi_1(M)$ be the fundamental group of $M$, and 
$H=\ab(G)$ its maximal torsion-free abelian quotient. 
Let $E_1(A_G)\subset \Z{H}$ be the Alexander ideal, and 
$\Delta_G=\gcd(E_1(A_G))$ the Alexander polynomial of $G$.  

We start with a lemma from \cite{DPS-imrn}, based on the work 
of Eisenbud--Neumann \cite{EN} and McMullen \cite{McM}. 

\begin{lemma} 
\label{lem:3mfd}
If either (1) $\partial M\ne \emptyset$ and $\chi (\partial M)=0$, or 
(2) $\partial M=\emptyset$ and $M$ is orientable, then $E_1(A_G)$ 
is almost principal (in the sense of Definition \ref{def:alexprinc}).
\end{lemma}

In the first case, $E_1(A_G)= I_H \cdot (\Delta_G)$, by \cite{EN}, 
while in the second case, $I^2_H\cdot (\Delta_G)\subseteq E_1(A_G)$, 
by \cite{McM}.  Proposition \ref{prop:v1 ap} now yields the 
following corollary, relating the first characteristic variety of $M$ 
to the zero set of the Alexander polynomial of $G=\pi_1(M)$. 

\begin{corollary}
\label{cor:v1 3m}
Under the above assumptions, 
\begin{romenum}
\item \label{3m1}
$\big( \VV_1(G) \cap \Hom(G,\C^{\times})^0\big) \setminus\set{1} 
= V(\Delta_G) \setminus\set{1}$.
\item  \label{3m2}
If, moreover, $H_1(G,\Z)$ is torsion-free, then 
$\VV_1(G)  \setminus\set{1} = V(\Delta_G) \setminus\set{1}$.
\end{romenum}
\end{corollary}

\begin{example}
\label{ex:links}
Let $L=(L_1,\dots ,L_n)$ be a link in $S^3$, and let 
$M=S^3 \setminus \bigcup_{i=1}^n N(L_i)$ be its exterior 
(i.e., the complement of an open tubular neighborhood of 
the link). Then $M$ is a compact, connected, orientable 
$3$-manifold, with $\partial M$ consisting of $n$ disjoint tori, 
so that $\chi(\partial M)=0$.  

Let $G=\pi_1(M)$ be the link group.  
By  Lemma \ref{lem:3mfd}, the Alexander ideal of $G$ is almost 
principal. Since $H_1(G,\Z)=\Z^n$, Corollary \ref{cor:v1 3m}\eqref{3m2} 
insures that $\VV_1(G)= V(\Delta_G)$, at least away from $1$.  
\end{example}

\subsection{Resonance varieties}
\label{subsec:res 3mfd}

Assume now $M$ is closed (i.e., $\partial M=\emptyset$) and 
orientable, and fix an orientation $[M]\in H^3(M,\Z)\cong\Z$. With 
this choice, the cup product on $M$ and Kronecker pairing 
determine an alternating $3$-form $\eta$ on $H^1(M,\Z)$, 
given by $\eta(x,y,z) = \langle x\cup y\cup z , [M]\rangle$. 
In turn, the cup-product map $\mu\colon H^1(M,\Z) \wedge 
H^1(M,\Z) \to H^2(M,\Z)$ is determined by the form $\eta$, 
via $\langle x\cup y , \gamma \rangle = \mu (x,y,z)$, 
where $z=\PD(\gamma)$ is the Poincar\'{e} dual of 
$\gamma \in H_2(M,\Z)$. 

Fix a basis $\{e_1,\dots ,e_n\}$ for $H^1(M,\C)$, 
and choose  $\{e^{\vee}_1,\dots ,e^{\vee}_n\}$ as basis 
for $H^2(M,\C)$, where $e^{\vee}_i$ denotes the Kronecker 
dual of $\PD(e_i)$. Writing 
$\mu(e_i, e_j)=\sum_{k=1}^{n} \mu_{ijk}  e^{\vee}_k$ as 
in \eqref{eq:mult}, we find that 
$\eta(e_i,e_j,e_k )= \mu_{ijk}$. 

By Proposition \ref{prop:res linalex}, the resonance 
variety $\RR_1(M)$ is the vanishing locus of the codimension 
$1$ minors of the $n\times n$ matrix $\Theta$ with entries 
$\Theta_{kj} =  \sum_{i=1}^n \mu_{ijk} x_i$.  
Since $\eta$ is an alternating form, $\Theta$ 
is a skew-symmetric matrix.  Using this fact, 
the following structure theorem is proved in \cite{DS}. 

\begin{theorem}[\cite{DS}]
\label{thm:res3}
Let $M$ be a closed, orientable $3$-manifold.  
Then:
\begin{romenum}
\item  \label{r1}
$H^1(M,\C)$ is not $1$-isotropic. 
\item \label{r2}
If $b_1(M)$ is even, then $\RR_1(M)=H^1(M,\C)$.
\item \label{r3}
If $\RR_1(M)$ contains a $0$-isotropic hyperplane, then 
$\RR_1(M)= H^1(M, \C)$.
\end{romenum}
\end{theorem}

\subsection{Thurston norm}
\label{subsec:thurston}

For each compact, orientable $3$-manifold $M$,  
Thurston defined in \cite{Th} a (semi-)norm on $H^1(M,\R)$, 
as follows. For a cohomology class $\phi\in H^1(M,\Z)$, 
set $\phinorm=\inf_{S}\{\chim (S)\}$, with the infimum 
taken over all properly embedded, compact surfaces 
$S$ representing the dual of $\phi$, and where $\chim(S)  = 
\sum_{i} \max \{-\chi(S_i),0\}$, after decomposing 
$S$ as a disjoint union of connected surfaces $S_i$. 

It turns out that the assignment $\phi \mapsto \|\phi \|$ 
is a seminorm on $H^1(M,\Z)$, which can be extended to the 
whole of $H^1(M,\R)$. Even though $\|\cdot\|$ vanishes on 
spherical and toroidal classes, it is simply called the {\em 
Thurston norm} of $M$. The unit ball in this norm, 
$B_T=\set{\phi \mid \phinorm \le 1}$, 
is a rational convex polyhedron, i.e., the intersection of 
finitely many rationally defined hyperplanes. 

A non-zero cohomology class $\phi\in H^1(M,\Z)$ is said to 
be {\em fibered}\/ if there is a smooth fibration $f\colon M\to S^1$ 
such that $\phi$ corresponds to $f_{*} \colon H_1(M,\Z)\to \Z$ 
under the isomorphism $H^1(M,\Z)\cong \Hom(H_1(M,\Z),\Z)$. 
In this case, $\phinorm=\chim(S)$, where $S=f^{-1}(1)$ is the 
fiber of the fibration. Given a fibered class $\phi$, there exists 
a (top-dimensional) face $F$  of $B_T$ so that $\phi$ belongs 
to the open cone $\R_{+} \cdot F^{\circ}$.  Such a face $F$ is 
called a {\em fibered face}, the reason being that 
each integral cohomology class in  
$\R_{+} \cdot F^{\circ}$ is fibered. 

\subsection{Alexander norm and Thurston norm}
\label{subsec:mcmullen}
In \cite{McM}, McMullen used the Alexander polynomial 
of $G=\pi_1(M)$ to define a (semi-)norm on $H^1(M,\R)$, 
called the {\em Alexander norm}.  If we write  
$\Delta_G=\sum n_i g_i$, with $n_i\in \Z\setminus\set{0}$ 
and $g_i \in \ab(G)$, then the norm of a class 
$\phi\in H^1(M,\R)$ is defined by $\| \phi \|_{A}=\sup \phi(g_i-g_j)$. 
The unit ball $B_A$ of the Alexander norm is, up to scaling, 
the polytope dual to the Newton polytope of $\Delta_G$. 

\begin{theorem}[\cite{McM}] 
\label{thm:McM}
Let $M$ be a compact, orientable $3$-manifold whose
boundary, if any, is a union of tori. Then: 
\begin{romenum}
\item   If $b_1(M)\ge 2$,  then $\| \phi \|_{A} \le \| \phi \|_{T}$, for all 
$\phi\in H^1(M,\R)$. 
\item   If $b_1(M)=1$,  then $\| \phi \|_{A} \le \| \phi \|_{T}+b_3(M)+1$, 
for all primitive classes $\phi\in H^1(M,\Z)$. 
\end{romenum}
Moreover, equality holds when $\phi\in H^1(M,\Z)$ is a fibered class, 
represented by a fibration $f\colon M\to S^1$ with $\chi(f^{-1}(1))\le 0$. 
\end{theorem}

McMullen's inequalities generalize the classical relation 
$\deg \Delta_K(t)\le 2\operatorname{genus}(K)$ for knots.  
In many cases, the Alexander and Thurston norms agree 
on all of $H^1(M,\R)$.  Yet even when $M$ fibers over the 
circle and $b_1(M)\ge 2$, the two norms may disagree, 
as examples of Dunfield \cite{Du} show. Nevertheless, 
as shown by Long \cite{Lo}, the two norms coincide for 
irreducible graph links in homology $3$-spheres. 

\subsection{BNS invariant}
\label{subsec:bns 3m}

In \cite{BNS}, Bieri, Neumann, and Strebel found a remarkable 
connection between the Thurston norm of $M$ and the 
BNS invariant of $G=\pi_1(M)$. We summarize their result, 
as follows.

\begin{theorem}[\cite{BNS}] 
\label{thm:bns 3m}
Let $M$ be a compact, connected $3$-manifold, with 
fundamental group $G=\pi_1(M)$.  Then, $\Sigma^1(G)$ 
is the union of open cones $\R_{+} \cdot F^{\circ}$,  
with $F$ running through the fibered faces of the 
Thurston unit ball $B_T$. Consequently:
\begin{romenum}
\item \label{bns-i} 
$\Sigma^1(G) = -\Sigma^1(G)$.
\item \label{bns-ii} 
$M$ fibers over the circle if and only if 
$\Sigma^1(G)\ne \emptyset$.
\end{romenum}
\end{theorem}

We exploit this theorem to give the example promised in 
\S\ref{subsec:res bound}, showing that the inclusion 
$\Sigma^{1}(G) \subseteq \RR_1(G,\R)^{\compl}$ 
may be strict, even when $G$ is $1$-formal and 
$\Sigma^1(G)=-\Sigma^1(G)$. 

\begin{example}
\label{ex:link2}
Let $L=(L_1,L_2)$ be a non-fibered, two-component link in 
$S^3$ with non-zero linking number. Let $X$ be the complement 
of $L$, and $G=\pi_1(X)$. Since $\lk(L_1,L_2)\ne 0$, we 
have a ring isomorphism $H^*(X,\Q)\cong H^*(T^2,\Q)$.  
Consequently, $X$ is a formal space, and 
$\RR_1(G,\R)=\set{0}$. 
On the other hand, since $G$ is a link group, 
$\Sigma^1(G)=-\Sigma^1(G)$, by 
Theorem \ref{thm:bns 3m}\eqref{bns-i}.
Furthermore, since $L$ is not a fibered link, 
$\Sigma^1(G)\ne\set{0}^{\compl}$, by 
Theorem \ref{thm:bns 3m}\eqref{bns-ii}.
\end{example}

Here is a another application. 

\begin{corollary}[\cite{PS-bns}]
\label{cor:3m formal}
Let $M$ be a closed, orientable $3$-manifold.   If $b_1(M)$ is 
even, and $G=\pi_1(M)$ is $1$-formal, then $M$ does not fiber 
over the circle.
\end{corollary}

Indeed, by Theorem \ref{thm:res3}\eqref{r2}, we must 
have $\RR_1(G,\R)=H^1(G,\R)$.  Hence, by Corollary 
\ref{cor:formal bns bound}, $\Sigma^1(G)=\emptyset$, 
and the conclusion follows from 
Theorem \ref{thm:bns 3m}\eqref{bns-ii}.

\subsection{K\"{a}hler $3$-manifold groups} 
\label{subsec:kahler 3mfd}

The following question was raised by \linebreak S.~Donaldson and 
W.~Goldman and in 1989, and independently by A.~Reznikov 
in 1993:  Which $3$-manifold groups are K\"{a}hler groups? 
In \cite{Re}, Reznikov obtained a deep restriction on certain groups 
lying at the intersection of these two classes of fundamental groups. 
In \cite{DS}, the question was answered in full generality.

\begin{theorem}[\cite{DS}]
\label{thm:3dk}
Let $G$ be the fundamental group of a closed $3$-manifold. 
Then $G$ is a K\"{a}hler group if and only if $G$ is a finite 
subgroup of $\operatorname{O}(4)$, acting freely on $S^3$.
\end{theorem}

The backward implication follows from classical work of J.-P.~Serre. 
In the case when $b_1(G)>0$, the forward implication can be readily  
explained, using the techniques outlined in \S\ref{subsec:res kahler} 
and \S\ref{subsec:res 3mfd}.  By passing to orientation double covers 
if necessary, we may assume $G=\pi_1(M)$, with $M$ a closed, 
orientable $3$-manifold. Then $G$ is still a K\"{a}hler group, 
and so $b_1(G)$ must be even (and positive).  From Theorem \ref{thm:res3}, 
we know that $\RR_1(G,\C)=H^1(G,\C)$, and $H^1(G,\C)$ is not 
$1$-isotropic. But, since $G$ is a K\"{a}hler group, this contradicts 
Theorem \ref{thm:res kahler}\eqref{rk4}.   

To deal with the remaining case, $b_1(G)=0$, one needs some 
deep facts from $3$-manifold theory, together with theorems of 
Reznikov and Fujiwara, relating the K\"{a}hler, respectively the 
$3$-manifold condition on a group to Kazhdan's property ($T$). 

A similar classification of quasi-K\"{a}hler, $3$-manifold 
groups seems to be beyond reach at the present time.  
Nevertheless, such groups are classified in 
\cite[Theorem 1.1]{DPS-qk}, but only at the level of 
Malcev completions, and under a formality assumption.

\begin{theorem}[\cite{DPS-qk}]
\label{thm:3d qk formal}
Let $G$ be the fundamental group of a closed, orientable
$3$-manifold. Assume $G$ is $1$-formal. Then the following 
are equivalent.
\begin{romenum}
\item \label{m1}
$\m(G)\cong \m(\pi_1(X))$, for some quasi-K\"{a}hler manifold $X$.

\item \label{m2}
$\m(G)\cong \m(\pi_1(M))$, where $M$ is either 
$S^3$, $\# ^n S^1\times S^2$, or $S^1\times \Sigma_g$.
\end{romenum}
\end{theorem}

In particular, if $G$ is a quasi-K\"{a}hler, $1$-formal, orientable 
$3$-manifold group, then its Malcev completion is isomorphic 
to that of either $1$, $F_n$, or $\Z\times \pi_1(\Sigma_g)$.

\subsection{Boundary manifolds of line arrangements}
\label{sec:bdry}
We conclude with a more detailed description of a special 
class of closed, orientable $3$-manifolds, arising in 
the context of hyperplane arrangements. 
Let $\A=\set{\ell_0,\dots,\ell_n}$ be an arrangement of lines 
in $\CP^2$. The {\em boundary manifold}\/ $M=M(\A)$ is 
obtained by taking the boundary of a regular neighborhood 
of $\bigcup_{i=0}^n \ell_i$ in $\CP^2$.  Such manifolds 
were investigated by Jiang and Yau \cite{JY93}, Westlund \cite{We}, 
and Hironaka \cite{Hi01}, and more recently, in \cite{CS06, CS08}. 

\begin{example}[\cite{CS06}]
\label{ex:bdry pencil}
Suppose $\A$ is a pencil of $n+1$ lines; then $M=\sharp^n S^1\times S^2$. 
On the other hand, if $\A$ is a near-pencil (i.e., a pencil of $n$ lines, 
together with another line in general position), then 
$M=S^1\times \Sigma_{n-1}$.
\end{example}

Suppose now $\A$ has $r$ non-transverse intersection points, i.e., 
points $F_J=\bigcap_{j \in J} \ell_j$ of multiplicity $\abs{J}\ge 3$. 
Blowing up $\CP^2$ at each of these points, we obtain an 
arrangement $\widetilde{\A}=\{L_0,\dots ,L_{n+r}\}$ in  
$\widetilde{\CP}^2 \cong \CP^2 \# r \overline{\CP}^2$, 
consisting of the proper transforms of the lines of $\A$, 
together with the exceptional lines arising from the blow-ups.  

This construction realizes the boundary manifold of $\A$ 
as a graph manifold, in the sense of Waldhausen.  The 
underlying graph $\G$ has vertex set $\cV(\G)$, with 
vertices in one-to-one correspondence with the lines 
of $\widetilde{\A}$: the vertex corresponding to $\ell_i$ is 
labeled $v_i$, while the vertex corresponding to $F_J$ 
is labeled $v_J$.  The graph $\G$ has 
an edge $e_{i,j}$ from $v_i$ to $v_j$, $i<j$, if the 
corresponding lines $\ell_i$ and $\ell_j$ are transverse, 
and an edge $e_{J,i}$ from $v_J$ to $v_i$ if 
$\ell_i \supset F_J$.  

Since $M$ is a graph manifold, the group $G=\pi_1(M)$ 
may be realized as the fundamental group of a graph of groups.  
As shown in \cite[Proposition 3.7]{CS08}, the resulting 
presentation for $G$ may be simplified to a commutator-relators 
presentation.

\begin{theorem}[\cite{CS08}]
\label{thm:alex poly arr}
Let $\A$ be an essential line arrangement in $\CP^2$, 
let $\G$ be the associated graph, and let $G$ be the 
fundamental group of the boundary manifold $M(\A)$.  
Then the Alexander polynomial of $G$ is 
\[
\Delta_G = \prod_{v \in \cV(\G)} (t_v-1)^{m_v-2}, 
\]
where $m_v$ denotes the degree of the vertex $v$, 
and $t_v=\prod_{i\in v} t_i$. 
\end{theorem}

\subsection{Jumping loci of boundary manifolds}
\label{subsec:jump bdry}
The first characteristic variety of the group 
$G=\pi_1(M(\A))$ can be easily computed from 
Theorem \ref{thm:alex poly arr}. First note 
that $\Delta_G(1)=0$.  Hence, by 
Corollary \ref{cor:v1 3m}\eqref{3m2}, 
\begin{equation}
\label{eq:v1 bdry}
\VV_1(G) = \bigcup_{v \in \cV(\G)\, :\, m_v\ge 3} 
\set{t_v-1=0}.  
\end{equation}

The resonance varieties of boundary manifolds 
were studied in detail in \cite{CS06, CS08}.   
As before, let $\A=\{\ell_0,\dots,\ell_n\}$ be an 
arrangement of lines in $\CP^2$. The first resonance 
variety of $G=\pi_1(M(\A))$ admits a particularly 
simple description:
\begin{equation}
\label{eq:r1 bdry}
\RR_1(G)=\begin{cases}
\C^n &\text{if $\A$ is a pencil,}\\
\C^{2(n-1)} &\text{if $\A$ is a near-pencil,}\\    
H^1(G,\C) &\text{otherwise.}
\end{cases}
\end{equation}

The higher-depth resonance varieties, though, can be much 
more complicated, as the following example extracted from 
\cite[\S6.9]{CS06} illustrates.   

\begin{example}
Let $\A$ be an arrangement of $5$ lines in $\CP^2$ in general 
position, and set $G=\pi_1(M(\A))$.  Then $H^1(G,\C)=\C^{10}$, 
and $\RR_7(G)=Q\times \{0\}$, where 
\[
Q=\{z\in \C^{6} \mid z_1z_6 -z_2z_5+z_3z_4=0\}, 
\]
which is an irreducible quadric, with an isolated singular point 
at $0$.  
On the other hand, formula \eqref{eq:v1 bdry} implies that 
$\VV_d(G)\subseteq \{1\}$, for all $d\ge 1$. Consequently, 
$TC_1(\VV_7(G))\ne \RR_7(G)$.  The tangent 
cone theorem now shows that $G$ is not $1$-formal. 
\end{example}

The above computation indicates that, for arrangements 
different from those in Example \ref{ex:bdry pencil}, 
the boundary manifold may well be non-formal.  As a 
matter of fact, we have the following result, which summarizes 
Theorem 9.7 from \cite{CS08}  and Proposition 4.7 from 
\cite{DPS-imrn}, and makes Theorem \ref{thm:3d qk formal} 
much more precise for this particular class of $3$-manifolds.

\begin{theorem}[\cite{CS08, DPS-imrn}] 
\label{thm:bdry}
Let $\A=\{\ell_0,\dots,\ell_n\}$ be an arrangement of lines in 
$\CP^2$, and let $M$ be the corresponding boundary manifold. 
The following are equivalent:
\begin{romenum}
\item  \label{f1} The manifold $M$ is formal. 
\item  \label{f2} The group $G=\pi_1(M)$ is $1$-formal.
\item  \label{f3'} The group $G$ is quasi-projective.
\item  \label{f3} The group $G$ is quasi-K\"{a}hler.
\item  \label{f4} $\A$ is either a pencil or a near-pencil. 
\item  \label{f5} $M$ is either $\sharp^n S^1\times S^2$ or 
$S^1\times \Sigma_{n-1}$.
\end{romenum}
\end{theorem}

The crucial implications here are \eqref{f2} $\Rightarrow$ 
 \eqref{f4}, which follows from Theorem \ref{thm:tcone}, 
together with formulas \eqref{eq:v1 bdry} and \eqref{eq:r1 bdry}, 
and \eqref{f3} $\Rightarrow$ \eqref{f4}, which makes use of  
Theorem \ref{thm:alex qk} and Theorem \ref{thm:alex poly arr}. 

%%%%  Bibliography %%%%%
\vspace*{2.0pc}

\newcommand{\arxiv}[1]
{\texttt{\href{http://arxiv.org/abs/#1}{arXiv:#1}}}

\newcommand{\arx}[1]
{\texttt{\href{http://arxiv.org/abs/#1}{arXiv:}}
\texttt{\href{http://arxiv.org/abs/#1}{#1}}}

\newcommand{\doi}[1]
{\texttt{\href{http://dx.doi.org/#1}{doi:#1}}}

\renewcommand{\MR}[1]
{\href{http://www.ams.org/mathscinet-getitem?mr=#1}{MR#1}}

\newcommand{\MRh}[2]
{\href{http://www.ams.org/mathscinet-getitem?mr=#1}{MR#1 (#2)}}

\end{document}